\documentclass[reqno]{amsart}
\usepackage[latin9]{luainputenc}
\usepackage{xcolor}
\usepackage{enumitem}
\usepackage{tipa}
\usepackage{amstext}
\usepackage{amsthm}
\usepackage{amssymb}
\usepackage{esint}
\PassOptionsToPackage{normalem}{ulem}
\usepackage{ulem}
\usepackage[unicode=true,pdfusetitle,
 bookmarks=true,bookmarksnumbered=false,bookmarksopen=false,
 breaklinks=false,pdfborder={0 0 0},pdfborderstyle={},backref=false,colorlinks=true]
 {hyperref}

\makeatletter

\providecolor{lyxadded}{rgb}{0,0,1}
\providecolor{lyxdeleted}{rgb}{1,0,0}

\DeclareRobustCommand{\lyxsout}[1]{\ifx\\#1\else\sout{#1}\fi}

\numberwithin{equation}{section}
\numberwithin{figure}{section}
\theoremstyle{plain}
\newtheorem{thm}{\protect\theoremname}
\theoremstyle{definition}
\newtheorem{defn}[thm]{\protect\definitionname}
\theoremstyle{plain}
\newtheorem{lem}[thm]{\protect\lemmaname}
\theoremstyle{remark}
\newtheorem{rem}[thm]{\protect\remarkname}
\theoremstyle{definition}
\newtheorem{example}[thm]{\protect\examplename}
\theoremstyle{plain}
\newtheorem{cor}[thm]{\protect\corollaryname}

\usepackage{pdfsync}
\usepackage{graphics}
\usepackage{epstopdf}
\usepackage[all]{xy}
\usepackage{amscd}
\usepackage{enumitem}
\usepackage{mathtools}
\usepackage{bbold}
\setlist[enumerate]{leftmargin=*,label=(\roman*),align=left}

\newcommand{\xyR}[1]{ \makeatletter
\xydef@\xymatrixrowsep@{#1} \makeatother} 
\newcommand{\xyC}[1]{ \makeatletter
\xydef@\xymatrixcolsep@{#1} \makeatother} 
\entrymodifiers={++[ ][F]} 

\newcommand{\ra}{\longrightarrow}

\newcommand{\field}[1]{\mathbb{#1}}
\newcommand{\R}{\field{R}} 
\newcommand{\N}{\field{N}} 

\newcommand{\eps}{\varepsilon} 
\renewcommand{\phi}{\varphi}
\newcommand{\diff}[1]{\ifmmode\mathchoice{\hbox{\rm d}#1}  
 {\hbox{\rm d}#1}  
 {\scalebox{0.75}{$\hbox{\rm d}#1$}}  
 {\scalebox{0.35}{$\hbox{\rm d}#1$}}  
 \fi} 

\newcommand{\abs}[2][\empty]{\ifx#1\empty\left|#2\right|%
\else#1\vert #2 #1\vert\fi}

\newcommand{\Coo}{\mbox{\ensuremath{\mathcal{C}}}^{\infty}} 

\newcommand{\st}[1]{{#1^\circ}} 

\newcommand{\Rtil}{\widetilde \R} 

\newcommand{\sint}[1]{\langle#1\rangle} 
\newcommand{\Eball}{B^{{\scriptscriptstyle \text{\rm E}}}} 

\newcommand{\fcmp}{\Subset_{\text{f}}}
\newcommand{\frontRise}[2]{\ifmmode\mathchoice{{\vphantom{#1}}^{\scalebox{0.6}{$#2$}}}  
 {{\vphantom{#1}}^{\scalebox{0.56}{$#2$}}}  
 {{\vphantom{#1}}^{\scalebox{0.47}{$#2$}}}  
 {{\vphantom{#1}}^{\scalebox{0.35}{$#2$}}}\fi} 
\newcommand{\RC}[1]{\frontRise{\R}{#1}\Rtil}
\newcommand{\rcrho}{\RC{\rho}}
\newcommand{\rti}{\RC{\rho}}

\newcommand{\hyperN}[1]{	\frontRise{\N}{#1}\widetilde{\N}}
\newcommand{\hypNr}{\hyperN{\rho}}
\newcommand{\hypNs}{\hyperN{\sigma}}
\newcommand{\nint}{\text{ni}}
\newcommand{\hyperlimarg}[3]{\mathchoice{\frontRise{\lim}{\raisebox{-0.05em}{$#1\hspace{-0.67em}$}}\lim_{#3\in \hyperN{#2}\,}}
{\frontRise{\lim}{#1\hspace{-0.25em}}\lim_{#3\in \hyperN{#2}\,}}
{\frontRise{\lim}{#1\hspace{-0.25em}}\lim_{#3\in \hyperN{#2}\,}}
{\frontRise{\lim}{#1\hspace{-0.25em}}\lim_{#3\in \hyperN{#2}\,}}}
\newcommand{\hyperlim}[2]{\hyperlimarg{#1}{#2}{n}}

\newcommand{\subzero}{\subseteq_{0}}
\newcommand{\dom}[1]{\text{dom}(#1)}

\newcommand{\sbpt}[1]{#1_{\text{\rm s}}}
\newcommand{\hyplimsuparg}[3]{\mathchoice{\frontRise{\limsup}{\raisebox{-0.05em}{$#1\hspace{-0.2em}$}}\limsup_{#3\in \hyperN{#2}\,}}
{\frontRise{\limsup}{#1\hspace{-0.23em}}\limsup_{#3\in \hyperN{#2}\,}}
{\frontRise{\limsup}{#1\hspace{-0.23em}}\limsup_{#3\in \hyperN{#2}\,}}
{\frontRise{\limsup}{#1\hspace{-0.23em}}\limsup_{#3\in \hyperN{#2}\,}}}
\newcommand{\hyplimsup}[2]{\hyplimsuparg{#1}{#2}{n}}
\newcommand{\hypliminfarg}[3]{\mathchoice{\frontRise{\liminf}{\raisebox{-0.05em}{$#1\hspace{-0.2em}$}}\liminf_{#3\in \hyperN{#2}\,}}
{\frontRise{\liminf}{#1\hspace{-0.23em}}\liminf_{#3\in \hyperN{#2}\,}}
{\frontRise{\liminf}{#1\hspace{-0.23em}}\liminf_{#3\in \hyperN{#2}\,}}
{\frontRise{\liminf}{#1\hspace{-0.23em}}\liminf_{#3\in \hyperN{#2}\,}}}
\newcommand{\hypliminf}[2]{\hypliminfarg{#1}{#2}{n}}

\newcommand{\ptind}{\displaystyle \mathop {\ldots\ldots\,}} 

\newcommand{\cinfty}{{\mathcal C}^\infty}

\newcommand{\sse}{\subseteq}

\makeatother

\providecommand{\corollaryname}{Corollary}
\providecommand{\definitionname}{Definition}
\providecommand{\examplename}{Example}
\providecommand{\lemmaname}{Lemma}
\providecommand{\remarkname}{Remark}
\providecommand{\theoremname}{Theorem}

\begin{document}
\title[Sup, inf and hyperlimits in the non-Archimedean ring of CGN]{Supremum, infimum and hyperlimits in the non-Archimedean ring of
Colombeau generalized numbers}
\author{A.~Mukhammadiev \and D.~Tiwari \and G.~Apaaboah \and P.~Giordano}
\thanks{A. Mukhammadiev has been supported by grant P30407 and P33538 of the
Austrian Science Fund FWF}
\address{\textsc{University of Vienna, Austria}}
\email{akbarali.mukhammadiev@univie.ac.at}
\thanks{D. Tiwari has been supported by grant P30407 and P33538 of the Austrian
Science Fund FWF}
\address{\textsc{University of Vienna, Austria}}
\email{diksha.tiwari@univie.ac.at}
\address{University Grenoble Alpes, France}
\email{apaaboag@etu.univ-grenoble-alpes.fr}
\thanks{P.~Giordano has been supported by grants P30407, P33538 and P34113
of the Austrian Science Fund FWF}
\address{\textsc{Wolfgang Pauli Institute, Vienna, Austria}}
\email{paolo.giordano@univie.ac.at}
\subjclass[2000]{46F-XX, 46F30, 26E30}
\keywords{Colombeau generalized numbers, non-Archimedean rings, generalized
functions.}
\begin{abstract}
It is well-known that the notion of limit in the sharp topology of
sequences of Colombeau generalized numbers $\widetilde{\mathbb{R}}$
does not generalize classical results. E.g.~the sequence $\frac{1}{n}\not\to0$
and a sequence $(x_{n})_{n\in\mathbb{N}}$ converges \emph{if} and
only if $x_{n+1}-x_{n}\to0$. This has several deep consequences,
e.g.~in the study of series, analytic generalized functions, or sigma-additivity
and classical limit theorems in integration of generalized functions.
The lacking of these results is also connected to the fact that $\widetilde{\mathbb{R}}$
is necessarily not a complete ordered set, e.g.~the set of all the
infinitesimals has neither supremum nor infimum. We present a solution
of these problems with the introduction of the notions of hypernatural
number, hypersequence, close supremum and infimum. In this way, we
can generalize all the classical theorems for the hyperlimit of a
hypersequence. The paper explores ideas that can be applied to other
non-Archimedean settings.
\end{abstract}

\maketitle

\section{Introduction}

A key concept of non-Archimedean analysis is that extending the real
field $\R$ into a ring containing infinitesimals and infinite numbers
could eventually lead to the solution of non trivial problems. This
is the case, e.g., of Colombeau theory, where nonlinear generalized
functions can be viewed as set-theoretical maps on domains consisting
of generalized points of the non-Archimedean ring $\Rtil$. This orientation
has become increasingly important in recent years and hence it has
led to the study of preliminary notions of $\Rtil$ (cf., e.g., \cite{OK,AJ,AFJ,ObVe08,AFJ09,Ar-Fe-Ju-Ob12,Ver10,LB-Gio17,GKV};
see below for a self-contained introduction to the ring of Colombeau
generalized numbers $\Rtil$).

In particular, the sharp topology on $\Rtil$ (cf., e.g., \cite{GK13b,S,S0}
and below) is the appropriate notion to deal with continuity of this
class of generalized functions and for a suitable concept of well-posedness.
This topology necessarily has to deal with balls having infinitesimal
radius $r\in\Rtil$, and thus $\frac{1}{n}\not\to0$ if $n\to+\infty$,
$n\in\N$, because we never have $\R_{>0}\ni\frac{1}{n}<r$ if $r$
is infinitesimal. Another unusual property related to the sharp topology
can be derived from the following inequalities (where $m\in\N$, $n\in\N_{\le m}$,
$r\in\Rtil_{>0}$ is an infinitesimal number, and $\left|x_{k+1}-x_{k}\right|\le r^{2}$)
\[
\left|x_{m}-x_{n}\right|\le\left|x_{m}-x_{m-1}\right|+\ldots+\left|x_{n+1}-x_{n}\right|\le(m-n)r^{2}<r,
\]
which imply that $(x_{n})_{n\in\N}\in\Rtil^{\N}$ is a Cauchy sequence
if and only if $\left|x_{n+1}-x_{n}\right|\to0$ (actually, this is
a well-known property of every ultrametric space, cf., e.g., \cite{Kob96,S}).
Naturally, this has several counter-intuitive consequences (arising
from differences with the classical theory) when we have to deal with
the study of series, analytic generalized functions, or sigma-additivity
and classical limit theorems in integration of generalized functions
(cf., e.g., \cite{PiScVa09,Ver08,TI}).

One of the aims of the present article is to solve this kind of counter-intuitive
properties so as to arrive at useful notions for the theory of generalized
functions. In order to settle this problem, it is important to generalize
the role of the net $(\eps)$, as used in Colombeau theory, into a
more general $\rho=(\rho_{\eps})\to0$ (which is called a \emph{gauge}),
and hence to generalize $\Rtil$ into some $\rcrho$ (see Def.~\ref{def:RCGN}).
We then introduce the set of \emph{hypernatural numbers} as
\[
\hypNr:=\left\{ [n_{\eps}]\in\rcrho\mid n_{\eps}\in\N\quad\forall\eps\right\} ,
\]
so that it is natural to expect that $\frac{1}{n}\to0$ in the sharp
topology if $n\to+\infty$ with $n\in\hypNr$, because now $n$ can
also take infinite values. The notion of sequence is therefore substituted
with that of \emph{hypersequence}, as a map $(x_{n})_{n}:\hypNs\ra\rcrho$,
where $\sigma$ is, generally speaking, another gauge. As we will
see, (cf.~example~\ref{exa:hyperlimits}) only in this way we are
able to prove e.g.~that $\frac{1}{\log n}\to0$ in $\rcrho$ as $n\in\hypNs$
but only for a suitable gauge $\sigma$ (depending on $\rho$), whereas
this limit does not exist if $\sigma=\rho$.

Finally, the notions of supremum and infimum are naturally linked
to the notion of limit of a monotonic (hyper)sequence. Being an ordered
set, $\rcrho$ already has a definition of, let us say, supremum as
least upper bound. However, as already preliminary studied and proved
by \cite{GarVer11}, this definition does not fit well with topological
properties of $\rcrho$ because generalized numbers $[x_{\eps}]\in\rcrho$
can actually jump as $\eps\to0^{+}$ (see Sec.~\ref{subsec:Archimedean-upper-bound}).
It is well known that in $\R$ we have $m=\sup(S)$ if and only if
$m$ is an upper bound of $S$ and
\begin{equation}
\forall r\in\R_{>0}\,\exists s\in S:\ m-r\le s.\label{eq:closedSupR}
\end{equation}
This could be generalized into the notion of \emph{close supremum}
in $\rcrho$, generalizing \cite{GarVer11}, that results into better
topological properties, see Sec.~\ref{sec:Supremum-and-Infimum}.
The ideas presented in the present article, which is self-contained,
can surely be useful to explore similar ideas in other non-Archimedean
settings, such as \cite{BeMa19,BeLuB15,Sha13,Palm,Kob96}.

\section{The Ring of Robinson Colombeau and the hypernatural numbers}

In this section, we introduce our non-Archimedean ring of scalars
and its subset of hypernatural numbers. For more details and proofs
about the basic notions introduced here, the reader can refer e.g.~to
\cite{C1,TI,GKOS}.

As we mentioned above, in order to accomplish the theory of hyperlimits,
it is important to generalize Colombeau generalized numbers by taking
an arbitrary asymptotic scale instead of the usual $\rho_{\eps}=\eps$:
\begin{defn}
\label{def:RCGN}Let $\rho=(\rho_{\eps})\in(0,1]^{I}$ be a net such
that $(\rho_{\eps})\to0$ as $\eps\to0^{+}$ (in the following, such
a net will be called a \emph{gauge}), then
\begin{enumerate}
\item $\mathcal{I}(\rho):=\left\{ (\rho_{\eps}^{-a})\mid a\in\R_{>0}\right\} $
is called the \emph{asymptotic gauge} generated by $\rho$.
\item If $\mathcal{P}(\eps)$ is a property of $\eps\in I$, we use the
notation $\forall^{0}\eps:\,\mathcal{P}(\eps)$ to denote $\exists\eps_{0}\in I\,\forall\eps\in(0,\eps_{0}]:\,\mathcal{P}(\eps)$.
We can read $\forall^{0}\eps$ as \emph{for $\eps$ small}.
\item We say that a net $(x_{\eps})\in\R^{I}$ \emph{is $\rho$-moderate},
and we write $(x_{\eps})\in\R_{\rho}$ if 
\[
\exists(J_{\eps})\in\mathcal{I}(\rho):\ x_{\eps}=O(J_{\eps})\text{ as }\eps\to0^{+},
\]
i.e., if 
\[
\exists N\in\N\,\forall^{0}\eps:\ |x_{\eps}|\le\rho_{\eps}^{-N}.
\]
\item Let $(x_{\eps})$, $(y_{\eps})\in\R^{I}$, then we say that $(x_{\eps})\sim_{\rho}(y_{\eps})$
if 
\[
\forall(J_{\eps})\in\mathcal{I}(\rho):\ x_{\eps}=y_{\eps}+O(J_{\eps}^{-1})\text{ as }\eps\to0^{+},
\]
that is if 
\[
\forall n\in\N\,\forall^{0}\eps:\ |x_{\eps}-y_{\eps}|\le\rho_{\eps}^{n}.
\]
This is a congruence relation on the ring $\R_{\rho}$ of moderate
nets with respect to pointwise operations, and we can hence define
\[
\RC{\rho}:=\R_{\rho}/\sim_{\rho},
\]
which we call \emph{Robinson-Colombeau ring of generalized numbers}.
This name is justified by \cite{Rob73,C1}: Indeed, in \cite{Rob73}
A.~Robinson introduced the notion of moderate and negligible nets
depending on an arbitrary fixed infinitesimal $\rho$ (in the framework
of nonstandard analysis); independently, J.F.~Colombeau, cf.~e.g.~\cite{C1}
and references therein, studied the same concepts without using nonstandard
analysis, but considering only the particular infinitesimal $(\eps)$.
\item In particular, if the gauge $\rho=(\rho_{\eps})$ is non-decreasing,
then we say that $\rho$ is a \emph{monotonic gauge}. Clearly, considering
a monotonic gauge narrows the class of moderate nets: e.g.~if $\lim_{\eps\to\frac{1}{k}}x_{\eps}=+\infty$
for all $k\in\N_{>0}$, then $(x_{\eps})\notin\R_{\rho}$ for any
monotonic gauge $\rho$.
\end{enumerate}
\end{defn}

In the following, $\rho$ will always denote a net as in Def.~\ref{def:RCGN},
even if we will sometimes omit the dependence on the infinitesimal
$\rho$, when this is clear from the context. We will also use other
directed sets instead of $I$: e.g.~$J\subseteq I$ such that $0$
is a closure point of $J$, or $I\times\N$. The reader can easily
check that all our constructions can be repeated in these cases.

We also recall that we write $[x_{\eps}]\le[y_{\eps}]$ if there exists
$(z_{\eps})\in\R^{I}$ such that $(z_{\eps})\sim_{\rho}0$ (we then
say that $(z_{\eps})$ is \emph{$\rho$-negligible}) and $x_{\eps}\le y_{\eps}+z_{\eps}$
for $\eps$ small. Equivalently, we have that $x\le y$ if and only
if there exist representatives $[x_{\eps}]=x$ and $[y_{\eps}]=y$
such that $x_{\eps}\le y_{\eps}$ for all $\eps$.

Although the order $\le$ is not total, we still have the possibility
to define the infimum $[x_{\eps}]\wedge[y_{\eps}]:=[\min(x_{\eps},y_{\eps})]$,
the supremum $[x_{\eps}]\vee[y_{\eps}]:=\left[\max(x_{\eps},y_{\eps})\right]$
of a finite number of generalized numbers. Henceforth, we will also
use the customary notation $\RC{\rho}^{*}$ for the set of invertible
generalized numbers, and we write $x<y$ to say that $x\le y$ and
$x-y\in\rcrho^{*}$. Our notations for intervals are: $[a,b]:=\{x\in\RC{\rho}\mid a\le x\le b\}$,
$[a,b]_{\R}:=[a,b]\cap\R$. Finally, we set $\diff{\rho}:=[\rho_{\eps}]\in\RC{\rho}$,
which is a positive invertible infinitesimal, whose reciprocal is
$\diff{\rho}^{-1}=[\rho_{\eps}^{-1}]$, which is necessarily a strictly
positive infinite number.

The following result is useful to deal with positive and invertible
generalized numbers. For its proof, see e.g.~\cite{AJ,AFJ09,TI,GKOS}.
\begin{lem}
\label{lem:mayer} Let $x\in\RC{\rho}$. Then the following are equivalent:
\begin{enumerate}
\item \label{enu:positiveInvertible}$x$ is invertible and $x\ge0$, i.e.~$x>0$.
\item \label{enu:strictlyPositive}For each representative $(x_{\eps})\in\R_{\rho}$
of $x$ we have $\forall^{0}\eps:\ x_{\eps}>0$.
\item \label{enu:greater-i_epsTom}For each representative $(x_{\eps})\in\R_{\rho}$
of $x$ we have $\exists m\in\N\,\forall^{0}\eps:\ x_{\eps}>\rho_{\eps}^{m}$.
\item \label{enu:There-exists-a}There exists a representative $(x_{\eps})\in\R_{\rho}$
of $x$ such that $\exists m\in\N\,\forall^{0}\eps:\ x_{\eps}>\rho_{\eps}^{m}$.
\end{enumerate}
\end{lem}

\subsection{The language of subpoints}

The following simple language allows us to simplify some proofs using
steps recalling the classical real field $\R$. We first introduce
the notion of \emph{subpoint}:
\begin{defn}
For subsets $J$, $K\subseteq I$ we write $K\subzero J$ if $0$
is an accumulation point of $K$ and $K\sse J$ (we read it as: $K$
\emph{is co-final in $J$}). Note that for any $J\subzero I$, the
constructions introduced so far in Def.~\ref{def:RCGN} can be repeated
using nets $(x_{\eps})_{\eps\in J}$. We indicate the resulting ring
with the symbol $\rcrho^{n}|_{J}$. More generally, no peculiar property
of $I=(0,1]$ will ever be used in the following, and hence all the
presented results can be easily generalized considering any other
directed set. If $K\subzero J$, $x\in\rcrho^{n}|_{J}$ and $x'\in\rcrho^{n}|_{K}$,
then $x'$ is called a \emph{subpoint} of $x$, denoted as $x'\subseteq x$,
if there exist representatives $(x_{\eps})_{\eps\in J}$, $(x'_{\eps})_{\eps\in K}$
of $x$, $x'$ such that $x'_{\eps}=x_{\eps}$ for all $\eps\in K$.
In this case we write $x'=x|_{K}$, $\dom{x'}:=K$, and the restriction
$(-)|_{K}:\rcrho^{n}\ra\rcrho^{n}|_{K}$ is a well defined operation.
In general, for $X\sse\rcrho^{n}$ we set $X|_{J}:=\{x|_{J}\in\rcrho^{n}|_{J}\mid x\in X\}$.
\end{defn}

In the next definition, we introduce binary relations that hold only
\emph{on subpoints}. Clearly, this idea is inherited from nonstandard
analysis, where cofinal subsets are always taken in a fixed ultrafilter.
\begin{defn}
Let $x$, $y\in\rcrho$, $L\subzero I$, then we say
\begin{enumerate}
\item $x<_{L}y\ :\iff\ x|_{L}<y|_{L}$ (the latter inequality has to be
meant in the ordered ring $\rcrho|_{L}$). We read $x<_{L}y$ as ``\emph{$x$
is less than $y$ on $L$}''.
\item $x\sbpt{<}y\ :\iff\ \exists L\subzero I:\ x<_{L}y$. We read $x\sbpt{<}y$
as ``\emph{$x$ is less than $y$ on subpoints''.}
\end{enumerate}
Analogously, we can define other relations holding only on subpoints
such as e.g.: $\sbpt{\in}$, $\sbpt{\le}$, $\sbpt{=}$, $\sbpt{\subseteq}$,
etc.
\end{defn}

\noindent For example, we have
\begin{align*}
x\le y\  & \iff\ \forall L\subzero I:\ x\le_{L}y\\
x<y\  & \iff\ \forall L\subzero I:\ x<_{L}y,
\end{align*}
the former following from the definition of $\le$, whereas the latter
following from Lem.~\ref{lem:mayer}. Moreover, if $\mathcal{P}\left\{ x_{\eps}\right\} $
is an arbitrary property of $x_{\eps}$, then
\begin{equation}
\neg\left(\forall^{0}\eps:\ \mathcal{P}\left\{ x_{\eps}\right\} \right)\ \iff\ \exists L\subzero I\,\forall\eps\in L:\ \neg\mathcal{P}\left\{ x_{\eps}\right\} .\label{eq:negation}
\end{equation}

Note explicitly that, generally speaking, relations on subpoints,
such as $\sbpt{\le}$ or $\sbpt{=}$, do not inherit the same properties
of the corresponding relations for points. So, e.g., both $\sbpt{=}$
and $\sbpt{\le}$ are not transitive relations.

The next result clarifies how to equivalently write a negation of
an inequality or of an equality using the language of subpoints.
\begin{lem}
\label{lem:negationsSubpoints}Let $x$, $y\in\rcrho$, then
\begin{enumerate}
\item \label{enu:neg-le}$x\nleq y\quad\Longleftrightarrow\quad x\sbpt{>}y$
\item \label{enu:negStrictlyLess}$x\not<y\quad\Longleftrightarrow\quad x\sbpt{\ge}y$
\item \label{enu:negEqual}$x\ne y\quad\Longleftrightarrow\quad x\sbpt{>}y$
or $x\sbpt{<}y$
\end{enumerate}
\end{lem}

\begin{proof}
\ref{enu:neg-le} $\Leftarrow$: The relation $x\sbpt{>}y$ means
$x|_{L}>y|_{L}$ for some $L\subzero I$. By Lem.~\ref{lem:mayer}
for the ring $\rcrho|_{L}$, we get $\forall^{0}\varepsilon\in L:\ x_{\varepsilon}>y_{\varepsilon}$,
where $x=[x_{\varepsilon}]$, $y=[y_{\varepsilon}]$ are any representatives
of $x$, $y$ resp. The conclusion follows by \eqref{eq:negation}

\ref{enu:neg-le} $\Rightarrow$: Take any representatives $x=[x_{\varepsilon}]$,
$y=[y_{\varepsilon}]$. The property
\[
\forall q\in\R_{>0}\,\forall^{0}\eps:\ x_{\eps}\le y_{\eps}+\rho_{\eps}^{q}
\]
for $q\to+\infty$ implies $x\le y$. We therefore have
\[
\exists q\in\R_{>0}\,\exists L\subzero I\,\forall\eps\in L:\ x_{\eps}>y_{\eps}+\rho_{\eps}^{q},
\]
i.e.~$x>_{L}y$.

\ref{enu:negStrictlyLess} $\Rightarrow$: We have two cases: either
$x-y$ is not invertible or $x\not\le y$. In the former case, the
conclusion follows from \cite[Thm.~1.2.39]{GKOS}. In the latter one,
it follows from \ref{enu:neg-le}.

\ref{enu:negStrictlyLess} $\Leftarrow$: By contradiction, if $x<y$
then $x=_{L}y$ for some $L\subzero I$, which contradicts the invertibility
of $x-y$.

\ref{enu:negEqual} $\Rightarrow$: By contradiction, assume that
$x\sbpt{\not>}y$ and $x\sbpt{\not<}y$. Then \ref{enu:neg-le} would
yield $x\le y$ and $y\le x$, and hence $x=y$. The opposite implication
directly follows by contradiction.
\end{proof}
Using the language of subpoints, we can write different forms of dichotomy
or trichotomy laws for inequality. The first form is the following
\begin{lem}
\label{lem:trich1st}Let $x$, $y\in\rcrho$, then
\begin{enumerate}
\item \label{enu:dichotomy}$x\le y$ or $x\sbpt{>}y$
\item \label{enu:strictDichotomy}$\neg(x\sbpt{>}y$ and $x\le y)$
\item \label{enu:trichotomy}$x=y$ or $x\sbpt{<}y$ or $x\sbpt{>}y$
\item \label{enu:leThen}$x\le y\ \Rightarrow\ x\sbpt{<}y$ or $x=y$
\item \label{enu:leSubpointsIff}$x\sbpt{\le}y\ \Longleftrightarrow\ x\sbpt{<}y$
or $x\sbpt{=}y$.
\end{enumerate}
\end{lem}

\begin{proof}
\ref{enu:dichotomy} and \ref{enu:strictDichotomy} follows directly
from Lem.~\ref{lem:negationsSubpoints}. To prove \ref{enu:trichotomy},
we can consider that $x\sbpt{>}y$ or $x\sbpt{\not>}y$. In the second
case, Lem.~\ref{lem:negationsSubpoints} implies $x\le y$. If $y\le x$
then $x=y$; otherwise, once again by Lem.~\ref{lem:negationsSubpoints},
we get $x\sbpt{<}y$. To prove \ref{enu:leThen}, assume that $x\le y$
but $x\sbpt{\not<}y$, then $x\ge y$ by Lem.~\ref{lem:negationsSubpoints}.\ref{enu:neg-le}
and hence $x=y$. The implication $\Leftarrow$ of \ref{enu:leSubpointsIff}
is trivial. On the other hand, if $x\sbpt{\le}y$ and $x\sbpt{\not<}y$,
then $y\le x$ from Lem.~\ref{lem:negationsSubpoints}.\ref{enu:neg-le},
and hence $x\sbpt{=}y$.
\end{proof}
\noindent As usual, we note that these results can also be trivially
repeated for the ring $\rcrho|_{L}$. So, e.g., we have $x\not\le_{L}y$
if and only if $\exists J\subzero L:\ x>_{J}y$, which is the analog
of Lem.~\ref{lem:negationsSubpoints}.\ref{enu:neg-le} for the ring
$\rcrho|_{L}$.

The second form of trichotomy (which for $\rcrho$ can be more correctly
named as \emph{quadrichotomy})\emph{ is} stated as follows:
\begin{lem}
\label{lem:trich2nd}Let $x=[x_{\eps}]$, $y=[y_{\eps}]\in\rcrho$,
then
\begin{enumerate}
\item \label{enu:quadrichotomyLessEq}$x\le y$ or $x\ge y$ or $\exists L\subzero I:\ L^{c}\subzero I,\ x\ge_{L}y\text{ and }x\le_{L^{c}}y$
\item \label{enu:fromL2noL}If for all $L\subzero I$ the following implication
holds
\begin{equation}
x\le_{L}y,\text{ or }x\ge_{L}y\ \Rightarrow\ \forall^{0}\eps\in L:\ \mathcal{P}\left\{ x_{\eps},y_{\eps}\right\} ,\label{eq:trichHp}
\end{equation}
then $\forall^{0}\eps:\ \mathcal{P}\left\{ x_{\eps},y_{\eps}\right\} $.
\item \label{enu:fromL2noLStrict}If for all $L\subzero I$ the following
implication holds
\begin{equation}
x<_{L}y,\text{ or }x>_{L}y\text{ or }x=_{L}y\ \Rightarrow\ \forall^{0}\eps\in L:\ \mathcal{P}\left\{ x_{\eps},y_{\eps}\right\} ,\label{eq:trichHpStrict}
\end{equation}
then $\forall^{0}\eps:\ \mathcal{P}\left\{ x_{\eps},y_{\eps}\right\} $.
\end{enumerate}
\end{lem}

\begin{proof}
\ref{enu:quadrichotomyLessEq}: if $x\not\le y$, then $x\sbpt{>}y$
from Lem.~\ref{lem:negationsSubpoints}.\ref{enu:neg-le}. Let $[x_{\eps}]=x$
and $[y_{\eps}]=y$ be two representatives, and set $L:=\{\eps\in I\mid x_{\eps}\ge y_{\eps}\}$.
The relation $x\sbpt{>}y$ implies that $L\subzero I$. Clearly, $x\ge_{L}y$
(but note that in general we cannot prove $x>_{L}y$). If $L^{c}\not\subzero I$,
then $(0,\eps_{o}]\subseteq L$ for some $\eps_{0}$, i.e.~$x\ge y$.
On the contrary, if $L^{c}\subzero I$, then $x\le_{L^{c}}y$.

\ref{enu:fromL2noL}: Property \ref{enu:quadrichotomyLessEq} states
that we have three cases. If $x_{\eps}\le y_{\eps}$ for all $\eps\le\eps_{0}$,
then it suffices to set $L:=(0,\eps_{0}]$ in \eqref{eq:trichHp}
to get the claim. Similarly, we can proceed if $x\ge y$. Finally,
if $x\ge_{L}y$ and $x\le_{L^{c}}y$, then we can apply \eqref{eq:trichHp}
both with $L$ and $L^{c}$ to obtain
\begin{align*}
\forall^{0}\eps & \in L:\ \mathcal{P}\left\{ x_{\eps},y_{\eps}\right\} \\
\forall^{0}\eps & \in L^{c}:\ \mathcal{P}\left\{ x_{\eps},y_{\eps}\right\} ,
\end{align*}
from which the claim directly follows.

\ref{enu:fromL2noLStrict}: By contradiction, assume
\begin{equation}
\forall\eps\in L:\ \neg\mathcal{P}\left\{ x_{\eps},y_{\eps}\right\} ,\label{eq:quadrContr}
\end{equation}
for some $L\subzero I$. We apply \ref{enu:quadrichotomyLessEq} to
the ring $\rcrho|_{L}$ to obtain the following three cases:
\begin{equation}
x\le_{L}y\text{ or }x\ge_{L}y\text{ or }\exists J\subzero L:\ J^{c}\subzero L,\ x\ge_{J}y\text{ and }x\le_{J^{c}}y.\label{eq:casesStrictTrich}
\end{equation}
If $x\le_{L}y$, by Lem.~\ref{lem:trich1st}.\ref{enu:leThen} for
the ring $\rcrho|_{L}$, this case splits into two sub-cases: $x=_{L}y$
or $\exists K\subzero L:\ x<_{K}y$. If the former holds, using \eqref{eq:trichHpStrict}
we get $\mathcal{P}\left\{ x_{\eps},y_{\eps}\right\} $ $\forall^{0}\eps\in L$,
which contradicts \eqref{eq:quadrContr}. If $x<_{K}y$, then $K\subzero I$
and we can apply \eqref{eq:quadrContr} with $K$ to get $\mathcal{P}\left\{ x_{\eps},y_{\eps}\right\} $
$\forall^{0}\eps\in K$, which again contradicts \eqref{eq:quadrContr}
because $K\subzero L$. Similarly we can proceed with the other three
cases stated in \eqref{eq:casesStrictTrich}.
\end{proof}
\noindent Property Lem.~\ref{lem:trich2nd}.\ref{enu:fromL2noL}
represents a typical replacement of the usual dichotomy law in $\R$:
for arbitrary $L\subzero I$, we can assume to have two cases: either
$x\le_{L}y$ or $x\ge_{L}y$. If in both cases we are able to prove
$\mathcal{P}\{x_{\eps},y_{\eps}\}$ for $\eps\in L$ small, then we
always get that this property holds for all $\eps$ small. Similarly,
we can use the strict trichotomy law stated in \ref{enu:fromL2noLStrict}.

\subsection{Inferior, superior and standard parts}

Other simple tools that we can use to study generalized numbers of
$\rti$ are the \emph{inferior }and \emph{superior parts }of a number.
Only in this section of the article, we assume that $\rho$ is a monotonic
gauge.
\begin{defn}
\label{def:infSupParts}Let $x=[x_{\eps}]\in\rcrho$ be a generalized
number, then:
\begin{enumerate}
\item If $\exists L\in\R:\ L\le x$, then $x_{\text{i}}:=\left[\inf_{e\in(0,\eps]}x_{e}\right]$
is called the \emph{inferior part of $x$.}
\item If $\exists U\in\R:\ x\le U$, then $x_{\text{s}}:=\left[\sup_{e\in(0,\eps]}x_{e}\right]$
is called the \emph{superior part of $x$.}
\end{enumerate}
Moreover, we set:
\begin{enumerate}[resume]
\item $\st{x_{\text{i}}}:=\liminf_{\eps\to0^{+}}x_{\eps}\in\R\cup\{\pm\infty\}$,
where $[x_{\eps}]=x$ is any representative of $x$, is called the
\emph{inferior standard part} of $x$. Note that if $\exists x_{\text{i}}$,
i.e.~if $x$ is finitely bounded from below, then $\st{(x_{\text{i}})}=\st{x_{\text{i}}}\in\R$
and $\st{x_{\text{i}}}\ge x_{\text{i}}$.
\item $\st{x_{\text{s}}}:=\limsup_{\eps\to0^{+}}x_{\eps}\in\R\cup\{\pm\infty\}$,
where $[x_{\eps}]=x$ is any representative of $x$, is called the
\emph{superior standard part} of $x$. Note that if $\exists x_{\text{s}}$,
i.e.~if $x$ is finitely bounded from above, then $\st{(x_{\text{s}})}=\st{x_{\text{s}}}\in\R$
and $\st{x_{\text{s}}}\le x_{\text{s}}$.
\end{enumerate}
\end{defn}

\noindent Note that, since $\rho=(\rho_{\eps})$ is non-decreasing,
if $[x'_{\eps}]=x$ is another representative, then for all $e\in(0,\eps]$,
we have $x'_{e}\le x_{e}+\rho_{e}^{n}\le x_{e}+\rho_{\eps}^{n}\le\rho_{\eps}^{n}+\sup_{e\in(0,\eps]}x_{e}$
and hence $\sup_{e\in(0,\eps]}x'_{e}\le\rho_{\eps}^{n}+\sup_{e\in(0,\eps]}x_{e}$.
This shows that inferior and superior parts, when they exist, are
well-defined. Moreover, if $(z_{\eps})$ is negligible, then $\limsup_{\eps\to0^{+}}\left(x_{\eps}+z_{\eps}\right)\le\limsup_{\eps\to0^{+}}x_{\eps}+0$,
which shows that $\st{x_{\text{s}}}$ is well-defined (similarly for
$\st{x_{\text{i}}}$ using super-additivitiy of $\liminf$).

Clearly, $x_{\text{i}}\le x\le x_{\text{s}}$ and $\st{x_{\text{i}}}\le\st{x_{\text{s}}}$.
We have that the generalized number $x$ is near-standard if and only
if $\st{x_{\text{i}}}=\st{x_{\text{s}}}=:\st{x}\in\R$; it is infinitesimal
if and only if $\exists\,\st{x}=0$; it is a positive infinite number
if and only if $\st{x_{\text{i}}}=\st{x_{\text{s}}}=:\st{x}=+\infty$
(the same for negative infinite numbers); it is a finite number if
and only if $\st{x_{\text{i}}},$ $\st{x_{\text{s}}}\in\R$. Finally,
there always exist $x'$, $x''\subseteq x$ such that $x'\approx\st{x_{\text{i}}}$
and $x''\approx\st{x_{\text{s}}}$, where $x\approx y$ means that
$x-y$ is infinitesimal (i.e.~$|x-y|\le r$ for all $r\in\R_{>0}$
or, equivalently, $\lim_{\eps\to0^{+}}x_{\eps}-y_{\eps}=0$ for all
$[x_{\eps}]=x$, $[y_{\eps}]=y$). Therefore, any generalized number
in $\rcrho$ is either finite or some of its subpoints is infinite;
in the former case, some of its subpoints is near standard.

\subsection{\label{subsec:Topologies}Topologies on $\RC{\rho}^{n}$}

On the $\RC{\rho}$-module $\RC{\rho}^{n}$ we can consider the natural
extension of the Euclidean norm, i.e.~$|[x_{\eps}]|:=[|x_{\eps}|]\in\RC{\rho}$,
where $[x_{\eps}]\in\RC{\rho}^{n}$. Even if this generalized norm
takes values in $\RC{\rho}$, it shares some essential properties
with classical norms: 
\begin{align*}
 & |x|=x\vee(-x)\\
 & |x|\ge0\\
 & |x|=0\Rightarrow x=0\\
 & |y\cdot x|=|y|\cdot|x|\\
 & |x+y|\le|x|+|y|\\
 & ||x|-|y||\le|x-y|.
\end{align*}

\noindent It is therefore natural to consider on $\RC{\rho}^{n}$
topologies generated by balls defined by this generalized norm and
a set of radii:
\begin{defn}
\label{def:setOfRadii}We say that $\mathfrak{R}$ is a \emph{set
of radii} if
\begin{enumerate}
\item $\mathfrak{R}\subseteq\RC{\rho}_{\ge0}^{*}$ is a non-empty subset
of positive invertible generalized numbers.
\item For all $r$, $s\in\mathfrak{R}$ the infimum $r\wedge s\in\mathfrak{R}$.
\item $k\cdot r\in\mathfrak{R}$ for all $r\in\mathfrak{R}$ and all $k\in\R_{>0}$.
\end{enumerate}
\noindent Moreover, if $\mathfrak{R}$ is a set of radii and $x$,
$y\in\RC{\rho}$, then:
\begin{enumerate}
\item We write $x<_{\mathfrak{R}}y$ if $\exists r\in\mathfrak{R}:\ r\le y-x$.
\item $B_{r}^{\mathfrak{R}}(x):=\left\{ y\in\RC{\rho}^{n}\mid\left|y-x\right|<_{\mathfrak{R}}r\right\} $
for any $r\in\mathfrak{R}$.
\item $\Eball_{\rho}(x):=\{y\in\R^{n}\mid|y-x|<\rho\}$, for any $\rho\in\R_{>0}$,
denotes an ordinary Euclidean ball in $\R^{n}$.
\end{enumerate}
\end{defn}

\noindent For example, $\RC{\rho}_{\ge0}^{*}$ and $\R_{>0}$ are
sets of radii.
\begin{lem}
\label{lem:<R}Let $\mathfrak{R}$ be a set of radii and $x$, $y$,
$z\in\RC{\rho}$, then
\begin{enumerate}
\item \label{enu:NoReflex}$\neg\left(x<_{\mathfrak{R}}x\right)$.
\item \label{enu:transitive}$x<_{\mathfrak{R}}y$ and $y<_{\mathfrak{R}}z$
imply $x<_{\mathfrak{R}}z$.
\item \label{enu:radiiArePositive}$\forall r\in\mathfrak{R}:\ 0<_{\mathfrak{R}}r$.
\end{enumerate}
\end{lem}

\noindent The relation $<_{\mathfrak{R}}$ has better topological
properties as compared to the usual strict order relation $x\le y$
and $x\ne y$ (a relation that we will therefore never use) because
of the following result:
\begin{thm}
\label{thm:intersectionBalls}The set of balls $\left\{ B_{r}^{\mathfrak{R}}(x)\mid r\in\mathfrak{R},\ x\in\RC{\rho}^{n}\right\} $
generated by a set of radii $\mathfrak{R}$ is a base for a topology
on $\RC{\rho}^{n}$.
\end{thm}

\noindent Henceforth, we will only consider the sets of radii $\RC{\rho}_{\ge0}^{*}=\rcrho_{>0}$
and $\R_{>0}$ and will use the simplified notation $B_{r}(x):=B_{r}^{\frak{\mathfrak{R}}}(x)$
if $\mathfrak{R}=\rcrho_{>0}$. The topology generated in the former
case is called \emph{sharp topology}, whereas the latter is called
\emph{Fermat topology}. We will call \emph{sharply open set} any open
set in the sharp topology, and \emph{large open set} any open set
in the Fermat topology; clearly, the latter is coarser than the former.
It is well-known (see e.g.~\cite{AFJ09,AJ,GK13b,GK15,TI} and references
therein) that this is an equivalent way to define the sharp topology
usually considered in the ring of Colombeau generalized numbers. We
therefore recall that the sharp topology on $\rcrho^{n}$ is Hausdorff
and Cauchy complete, see e.g.~\cite{AFJ09,GK15}.

\subsection{Open, closed and bounded sets generated by nets}

A natural way to obtain sharply open, closed and bounded sets in $\RC{\rho}^{n}$
is by using a net $(A_{\eps})$ of subsets $A_{\eps}\subseteq\R^{n}$.
We have two ways of extending the membership relation $x_{\eps}\in A_{\eps}$
to generalized points $[x_{\eps}]\in\RC{\rho}^{n}$ (cf.\ \cite{ObVe08,GKV}):
\begin{defn}
\label{def:internalStronglyInternal}Let $(A_{\eps})$ be a net of
subsets of $\R^{n}$, then
\begin{enumerate}
\item $[A_{\eps}]:=\left\{ [x_{\eps}]\in\RC{\rho}^{n}\mid\forall^{0}\eps:\,x_{\eps}\in A_{\eps}\right\} $
is called the \emph{internal set} generated by the net $(A_{\eps})$.
\item Let $(x_{\eps})$ be a net of points of $\R^{n}$, then we say that
$x_{\eps}\in_{\eps}A_{\eps}$, and we read it as $(x_{\eps})$ \emph{strongly
belongs to $(A_{\eps})$}, if
\begin{enumerate}
\item $\forall^{0}\eps:\ x_{\eps}\in A_{\eps}$.
\item If $(x'_{\eps})\sim_{\rho}(x_{\eps})$, then also $x'_{\eps}\in A_{\eps}$
for $\eps$ small.
\end{enumerate}
\noindent Moreover, we set $\sint{A_{\eps}}:=\left\{ [x_{\eps}]\in\RC{\rho}^{n}\mid x_{\eps}\in_{\eps}A_{\eps}\right\} $,
and we call it the \emph{strongly internal set} generated by the net
$(A_{\eps})$.
\item We say that the internal set $K=[A_{\eps}]$ is \emph{sharply bounded}
if there exists $M\in\RC{\rho}_{>0}$ such that $K\subseteq B_{M}(0)$.
\item Finally, we say that the net $(A_{\eps})$ is \emph{sharply bounded
}if there exists $N\in\R_{>0}$ such that $\forall^{0}\eps\,\forall x\in A_{\eps}:\ |x|\le\rho_{\eps}^{-N}$.
\end{enumerate}
\end{defn}

\noindent Therefore, $x\in[A_{\eps}]$ if there exists a representative
$[x_{\eps}]=x$ such that $x_{\eps}\in A_{\eps}$ for $\eps$ small,
whereas this membership is independent from the chosen representative
in case of strongly internal sets. An internal set generated by a
constant net $A_{\eps}=A\subseteq\R^{n}$ will simply be denoted by
$[A]$.

The following theorem (cf.~\cite{ObVe08,GKV} for the case $\rho_{\eps}=\eps$,
and \cite{TI} for an arbitrary gauge) shows that internal and strongly
internal sets have dual topological properties:
\begin{thm}
\noindent \label{thm:strongMembershipAndDistanceComplement}For $\eps\in I$,
let $A_{\eps}\subseteq\R^{n}$ and let $x_{\eps}\in\R^{n}$. Then
we have
\begin{enumerate}
\item \label{enu:internalSetsDistance}$[x_{\eps}]\in[A_{\eps}]$ if and
only if $\forall q\in\R_{>0}\,\forall^{0}\eps:\ d(x_{\eps},A_{\eps})\le\rho_{\eps}^{q}$.
Therefore $[x_{\eps}]\in[A_{\eps}]$ if and only if $[d(x_{\eps},A_{\eps})]=0\in\RC{\rho}$.
\item \label{enu:stronglyIntSetsDistance}$[x_{\eps}]\in\sint{A_{\eps}}$
if and only if $\exists q\in\R_{>0}\,\forall^{0}\eps:\ d(x_{\eps},A_{\eps}^{\text{c}})>\rho_{\eps}^{q}$,
where $A_{\eps}^{\text{c}}:=\R^{n}\setminus A_{\eps}$. Therefore,
if $(d(x_{\eps},A_{\eps}^{\text{c}}))\in\R_{\rho}$, then $[x_{\eps}]\in\sint{A_{\eps}}$
if and only if $[d(x_{\eps},A_{\eps}^{\text{c}})]>0$.
\item \label{enu:internalAreClosed}$[A_{\eps}]$ is sharply closed.
\item \label{enu:stronglyIntAreOpen}$\sint{A_{\eps}}$ is sharply open.
\item \label{enu:internalGeneratedByClosed}$[A_{\eps}]=\left[\text{\emph{cl}}\left(A_{\eps}\right)\right]$,
where $\text{\emph{cl}}\left(S\right)$ is the closure of $S\subseteq\R^{n}$.
\item \label{enu:stronglyIntGeneratedByOpen}$\sint{A_{\eps}}=\sint{\text{\emph{int}\ensuremath{\left(A_{\eps}\right)}}}$,
where $\emph{int}\left(S\right)$ is the interior of $S\subseteq\R^{n}$.
\end{enumerate}
\end{thm}

\noindent For example, it is not hard to show that the closure in
the sharp topology of a ball of center $c=[c_{\eps}]$ and radius
$r=[r_{\eps}]>0$ is 
\begin{equation}
\overline{B_{r}(c)}=\left\{ x\in\rti^{d}\mid\left|x-c\right|\le r\right\} =\left[\overline{\Eball_{r_{\eps}}(c_{\eps})}\right],\label{eq:closureBall}
\end{equation}
whereas
\[
B_{r}(c)=\left\{ x\in\rti^{d}\mid\left|x-c\right|<r\right\} =\sint{\Eball_{r_{\eps}}(c_{\eps})}.
\]

Using internal sets and adopting ideas similar to those used in proving
Lem.~\ref{lem:trich2nd}, we also have the following form of dichotomy
law:
\begin{lem}
\label{lem:dichotIntSets}For $\eps\in I$, let $A_{\eps}\subseteq\R^{n}$
and let $x=[x_{\eps}]\in\rcrho^{n}$. Then we have:
\begin{enumerate}
\item \label{enu:trichIntSets}$x\in[A_{\eps}]$ or $x\in[A_{\eps}^{c}]$
or $\exists L\subzero I:\ L^{c}\subzero I,\ x\in_{L}[A_{\eps}],\ x\in_{L^{c}}[A_{\eps}^{c}]$
\item \label{enu:dichIntSets}If for all $L\subzero I$ the following implication
holds
\[
x\in_{L}[A_{\eps}]\text{ or }x\in_{L}[A_{\eps}^{c}]\ \Rightarrow\ \forall^{0}\eps\in L:\ \mathcal{P}\{x_{\eps}\},
\]
then $\forall^{0}\eps:\ \mathcal{P}\{x_{\eps}\}$.
\end{enumerate}
\end{lem}

\begin{proof}
\ref{enu:trichIntSets}: If $x\notin[A_{\eps}^{c}]$, then $x_{\eps}\in A_{\eps}$
for all $\eps\in K$ and for some $K\subzero I$. Set $L:=\left\{ \eps\in I\mid x_{\eps}\in A_{\eps}\right\} $,
so that $K\subseteq L\subzero I$. We have $x\in_{L}[A_{\eps}]$.
If $L^{c}\not\subzero I$, then $(0,\eps_{0}]\subseteq L$ for some
$\eps_{0}$, i.e.~$x\in[A_{\eps}]$. On the contrary, if $L^{c}\subzero I$,
then $x\in_{L^{c}}[A_{\eps}^{c}]$.

\ref{enu:dichIntSets}: We can proceed as in the proof of Lem.~\ref{lem:trich2nd}.\ref{enu:fromL2noL}
using \ref{enu:trichIntSets}.
\end{proof}

\section{Hypernatural numbers}

We start by defining the set of hypernatural numbers in $\rcrho$
and the set of $\rho$-moderate nets of natural numbers:
\begin{defn}
\label{def:hypernatural}We set
\begin{enumerate}
\item $\hypNr:=\left\{ [n_{\eps}]\in\rcrho\mid n_{\eps}\in\N\quad\forall\eps\right\} $
\item $\N_{\rho}:=\left\{ (n_{\eps})\in\R_{\rho}\mid n_{\eps}\in\N\quad\forall\eps\right\} .$
\end{enumerate}
\end{defn}

\noindent Therefore, $n\in\hypNr$ if and only if there exists $(x_{\eps})\in\R_{\rho}$
such that $n=[\text{int}(|x_{\eps}|)]$. Clearly, $\N\subset\hypNr$.
Note that the integer part function $\text{int}(-)$ is not well-defined
on $\rcrho$. In fact, if $x=1=\left[1-\rho_{\eps}^{1/\eps}\right]=\left[1+\rho_{\eps}^{1/\eps}\right]$,
then $\text{int}\left(1-\rho_{\eps}^{1/\eps}\right)=0$ whereas $\text{int}\left(1+\rho_{\eps}^{1/\eps}\right)=1$,
for $\eps$ sufficiently small. Similar counter examples can be set
for floor and ceiling functions. However, the nearest integer function
is well defined on $\hypNr$, as proved in the following
\begin{lem}
\label{lem:nearestInt}Let $(n_{\eps})\in\N_{\rho}$ and $(x_{\eps})\in\R_{\rho}$
be such that $[n_{\eps}]=[x_{\eps}]$. Let $\text{\emph{rpi}}:\R\ra\N$
be the function rounding to the nearest integer with tie breaking
towards positive infinity, i.e.~$\text{rpi}(x)=\lfloor x+\frac{1}{2}\rfloor$.
Then $\text{\emph{rpi}}(x_{\eps})=n_{\eps}$ for $\eps$ small. The
same result holds using $\text{\emph{rni}}:\R\ra\N$, the function
rounding half towards $-\infty$.
\end{lem}

\begin{proof}
We have $\text{rpi}(x)=\lfloor x+\frac{1}{2}\rfloor$, where $\lfloor-\rfloor$
is the floor function. For $\eps$ small, $\rho_{\eps}<\frac{1}{2}$
and, since $[n_{\eps}]=[x_{\eps}]$, always for $\eps$ small, we
also have $n_{\eps}-\rho_{\eps}+\frac{1}{2}<x_{\eps}+\frac{1}{2}<n_{\eps}+\rho_{\eps}+\frac{1}{2}$.
But $n_{\eps}\le n_{\eps}-\rho_{\eps}+\frac{1}{2}$ and $n_{\eps}+\rho_{\eps}+\frac{1}{2}<n_{\eps}+1$.
Therefore $\lfloor x_{\eps}+\frac{1}{2}\rfloor=n_{\eps}$. An analogous
argument can be applied to $\text{rni}(-)$.
\end{proof}
\noindent Actually, this lemma does not allow us to define a \emph{nearest
integer }function $\nint{}:\hypNr\ra\N_{\rho}$ as $\nint{([x_{\eps}])}:=\text{rpi}(x_{\eps})$
because if $[x_{\eps}]=[n_{\eps}]$, the equality $n_{\eps}=\text{rpi}(x_{\eps})$
holds only for $\eps$ small. A simpler approach is to choose a representative
$(n_{\eps})\in\N_{\rho}$ for each $x\in\hypNr$ and to define $\nint{(x)}:=(n_{\eps})$.
Clearly, we must consider the net $\left(\nint{(x)}_{\eps}\right)$
only for $\eps$ small, such as in equalities of the form $x=\left[\nint{(x)}_{\eps}\right]$.
This is what we do in the following
\begin{defn}
\label{def:nint}The \emph{nearest integer function} $\nint(-)$ is
defined by:
\begin{enumerate}
\item $\nint:\hypNr:\ra\N_{\rho}$
\item If $[x_{\eps}]\in\hypNr$ and $\nint\left([x_{\eps}]\right)=(n_{\eps})$
then $\forall^{0}\eps:\ n_{\eps}=\text{rpi}(x_{\eps})$.
\end{enumerate}
In other words, if $x\in\hypNr$, then $x=\left[\nint(x)_{\eps}\right]$
and $\nint(x)_{\eps}\in\N$ for all $\eps$. Another possibility is
to formulate Lem.~\ref{lem:nearestInt} as
\[
[x_{\eps}]\in\hypNr\quad\iff\quad[x_{\eps}]=[\text{rpi}(x_{\eps})].
\]
Therefore, without loss of generality we may always suppose that $x_{\eps}\in\N$
whenever $[x_{\eps}]\in\hypNr$.
\end{defn}

\begin{rem}
~
\begin{enumerate}
\item $\hyperN{\sigma}$, with the order $\le$ induced by $\RC{\sigma}$,
is a directed set; it is closed with respect to sum and product although
recursive definitions using $\hyperN{\sigma}$ are not possible.
\item In $\hyperN{\sigma}$ we can find several chains (totally ordered
subsets) such as: $\N$, $\N\cdot[\text{int}(\rho_{\eps}^{-k})]$
for a fixed $k\in\N$, $\{[\text{int}(\rho_{\eps}^{-k})]\mid k\in\N\}$.
\item Generally speaking, if $m$, $n\in\hyperN{\rho}$, $m^{n}\notin\hyperN{\rho}$
because the net $\left(m_{\eps}^{n_{\eps}}\right)$ can grow faster
than any power $(\rho_{\eps}^{-K})$. However, if we take two gauges
$\sigma$, $\rho$ satisfying $\sigma\le\rho$, using the net $\left(\sigma_{\eps}^{-1}\right)$
we can measure infinite nets that grow faster than $(\rho_{\eps}^{-K})$
because $\sigma_{\eps}^{-1}\ge\rho_{\eps}^{-1}$ for $\eps$ small.
Therefore, we can take $m$, $n\in\hyperN{\sigma}$ such that $\left(\nint{(m)}_{\eps}\right)$,
$\left(\nint{(n)}_{\eps}\right)\in\R_{\rho}$; we think at $m$, $n$
as $\sigma$-hypernatural numbers growing at most polynomially with
respect to $\rho$. Then, it is not hard to prove that if $\rho$
is an arbitrary gauge, and we consider the auxiliary gauge $\sigma_{\eps}:=\rho_{\eps}^{e^{1/\rho_{\eps}}}$.
then $m^{n}\in\hyperN{\sigma}$.
\item If $m\in\hypNr$, then $1^{m}:=\left[\left(1+z_{\eps}\right)^{m_{\eps}}\right]$,
where $(z_{\eps})$ is $\rho$-negligible, is well defined and $1^{m}=1$.
In fact, $\log(1+z_{\eps})^{m_{\eps}}$ is asymptotically equal to
$m_{\eps}z_{\eps}\to0$, and this shows that $\left(\left(1+z_{\eps}\right)^{m_{\eps}}\right)$
is moderate. Finally, $\left|(1+z_{\eps})^{m_{\eps}}-1\right|\le\left|z_{\eps}\right|m_{\eps}(1+z_{\eps})^{m_{\eps}-1}$
by the mean value theorem.
\end{enumerate}
\end{rem}

\section{\label{sec:Supremum-and-Infimum}Supremum and Infimum in $\rcrho$}

To solve the problems we explained in the introduction of this article,
it is important to generalize at least two main existence theorems
for limits: the Cauchy criterion and the existence of a limit of a
bounded monotone sequence. The latter is clearly related to the existence
of supremum and infimum, which cannot be always guaranteed in the
non-Archimedean ring $\rcrho$. As we will see more clearly later
(see also \cite{GarVer11}), to arrive at these existence theorems,
the notion of supremum, i.e.~the least upper bound, is not the correct
one. More appropriately, we can associate a notion of close supremum
(and close infimum) to every topology generated by a set of radii
(see Def.~\ref{def:setOfRadii}).
\begin{defn}
Let $\mathfrak{R}$ be a set of radii and let $\tau$ be the topology
on $\rcrho$ generated by $\mathfrak{R}$. Let $P\subseteq\rcrho$,
then we say that $\tau$ \emph{separates points of} $P$ if
\[
\forall p,q\in P:\ p\mathbb{\neq}q\ \Rightarrow\ \exists A,B\in\tau:\ p\in A,\ q\in B,\ A\cap B=\emptyset,
\]
i.e.~if $P$ with the topology induced by $\tau$ is Hausdorff.
\end{defn}

\begin{defn}
\label{def:topSup}Let $\tau$ be a topology on $\rcrho$ generated
by a set of radii $\mathfrak{R}$ that separates points of $P\subseteq\rcrho$
and let $S\subseteq\rcrho$. Then, we say that $\sigma$ is $\left(\tau,P\right)$\emph{-supremum
of} $S$ if
\begin{enumerate}
\item $\sigma\in P$;
\item \label{enu:1defSup}$\forall s\in S:\ s\leq\sigma$;
\item \label{enu:2defSupGen}$\sigma$ is a point of closure of $S$ in
the topology $\tau$, i.e.~if $\forall A\in\tau:\ \sigma\in A\ \Rightarrow\ \exists\bar{s}\in S\cap A$.
\end{enumerate}
Similarly, we say that $\iota$ is $(\tau,P)$-\emph{infimum of }$S$
if
\begin{enumerate}
\item $\iota\in P$;
\item $\forall s\in S:\ \iota\le s$;
\item $\iota$ is a point of closure of $S$ in the topology $\tau$, i.e.~if
$\forall A\in\tau:\ \iota\in A\ \Rightarrow\ \exists\bar{s}\in S\cap A$.
\end{enumerate}
In particular, if $\tau$ is the sharp topology and $P=\rcrho$, then
following \cite{GarVer11}, we simply call the $(\tau,P)$-supremum,
the \emph{close supremum} (the adjective close will be omitted if
it will be clear from the context) or the \emph{sharp supremum} if
we want to underline the dependency on the topology. Analogously,
if $\tau$ is the Fermat topology and $P=\R$, then we call the $(\tau,P)$-supremum
the \emph{Fermat supremum}. Note that \ref{enu:2defSupGen} implies
that if $\sigma$ is $(\tau,P)$-supremum of $S$, then necessarily
$S\ne\emptyset$.

\end{defn}

\begin{rem}
\label{rem:sup}\ 
\begin{enumerate}
\item \label{rem:sharpSupremum}Let $S\subseteq\rcrho$, then from Def.~\ref{def:setOfRadii}
and Thm.~\ref{thm:intersectionBalls} we can prove that $\sigma$
is the $(\tau,P)$-supremum of $S$ if and only if
\begin{enumerate}[label=(\alph*)]
\item \label{enu:1defSupbis}$\forall s\in S:\ s\leq\sigma$;
\item \label{enu:2defSup}$\forall r\in\mathfrak{R}\,\exists\bar{s}\in S:\ \sigma-r\leq\bar{s}$.
\end{enumerate}
In particular, for the sharp supremum, \ref{enu:2defSup} is equivalent
to
\begin{equation}
\forall q\in\N\,\exists\bar{s}\in S:\ \sigma-\diff{\rho}^{q}\le\bar{s}.\label{eq:2defSup}
\end{equation}
In the following of this article, we will also mainly consider the
sharp topology and the corresponding notions of sharp supremum and
infimum.
\item If there exists the sharp supremum $\sigma$ of $S\subseteq\rcrho$
and $\sigma\notin S$, then from \eqref{eq:2defSup} it follows that
$S$ is necessarily an infinite set. In fact, applying \eqref{eq:2defSup}
with $q_{1}:=1$ we get the existence of $\bar{s}_{1}\in S$ such
that $\sigma-\diff{\rho}^{q_{1}}<\bar{s_{1}}$. We have $\bar{s}_{1}\ne\sigma$
because $\sigma\notin S$. Hence, Lem.~\ref{lem:negationsSubpoints}.\ref{enu:negEqual}
and Def.~\ref{def:topSup}.\ref{enu:1defSup} yield that $\bar{s}_{1}\sbpt{<}\sigma$.
Therefore, $\sigma-\bar{s}_{1}\sbpt{\ge}\diff{\rho}^{q_{2}}$ for
some $q_{2}>q_{1}$. Applying again \eqref{eq:2defSup} we get $\sigma-\diff{\rho}^{q_{2}}<\bar{s}_{2}$
for some $\bar{s}_{2}\in S\setminus\{\bar{s}_{1}\}$. Recursively,
this process proves that $S$ is infinite. On the other hand, if $S=\{s_{1},\ldots,s_{n}\}$
and $s_{i}=[s_{i\eps}]$, then $\sup\left(\left[\{s_{1\eps},\ldots,s_{n\eps}\}\right]\right)=s_{1}\vee\ldots\vee s_{n}$.
In fact, $s_{1}\vee\ldots\vee s_{n}=\left[\max_{i=1,\ldots,n}s_{n\eps}\right]\in[\{s_{1\eps},\ldots,s_{n\eps}\}]$.
\item If $\exists\sup(S)=\sigma$, then there also exists the $\sup(\text{interl}(S))=\sigma$,
where (see \cite{ObVe08}) we recall that
\[
\text{interl}(S):=\left\{ \sum_{j=1}^{m}e_{S_{j}}s_{j}\mid m\in\N,\ S_{j}\subzero I,\ s_{j}\in S\ \forall j\right\} ,\ e_{S}:=\left[1_{S}\right]\in\rti
\]
($1_{S}$ is the characteristic function of $S\subseteq I$). This
follows from $S\subseteq\text{interl}(S)$. Vice versa, if $\exists\sup(\text{interl}(S))=\sigma$
and $\text{interl}(S)\subseteq S$ (e.g.~if $S$ is an internal or
strongly internal set), then also $\exists\sup(S)=\sigma$.
\end{enumerate}
\end{rem}

\begin{thm}
There is at most one sharp supremum of $S$, which is denoted by $\sup(S)$.
\end{thm}

\begin{proof}
Assume that $\sigma_{1}$ and $\sigma_{2}$ are supremum of $S$.
That is Def.~\ref{def:topSup}.\ref{enu:1defSup} and \eqref{eq:2defSup}
hold both for $\sigma_{1}$, $\sigma_{2}$. Then, for all fixed $q\in\N$,
there exists $\bar{s}_{2}\in S$ such that $\sigma_{2}-\diff{\rho}^{q}\le\bar{s}_{2}$.
Hence $\bar{s}_{2}\leq\sigma_{1}$ because $\bar{s}_{2}\in S$. Analogously,
we have that $\sigma_{1}-\diff{\rho}^{q}\le\bar{s}_{1}\leq\sigma_{2}$
for some $\bar{s}_{1}\in S$. Therefore, $\sigma_{2}-\diff{\rho}^{q}\leq\sigma_{1}\leq\sigma_{2}+\diff{\rho}^{q}$,
and this implies $\sigma_{1}=\sigma_{2}$ since $q\in\N$ is arbitrary.
\end{proof}
In \cite{GarVer11}, the notation $\overline{\sup}(S)$ is used for
the close supremum. On the other hand, we will \emph{never} use the
notion of supremum as least upper bound. For these reasons, we prefer
to use the simpler notation $\sup(S)$. Similarly, we use the notation
$\inf(S)$ for the close (or sharp) infimum. From Rem\@.~\ref{rem:sup}.\ref{enu:1defSupbis}
and \ref{enu:2defSup} it follows that
\begin{equation}
\inf(S)=-\sup(-S)\label{eq:infSup}
\end{equation}
in the sense that the former exists if and only if the latter exists
and in that case they are equal. For this reason, in the following
we only study the supremum.
\begin{example}
\label{exa:supInf}~
\begin{enumerate}
\item Let $K=[K_{\eps}]\fcmp\rcrho$ be a functionally compact set (cf.~\cite{GK15}),
i.e.~$K\subseteq B_{M}(0)$ for some $M\in\rcrho_{>0}$ and $K_{\eps}\Subset\R$
for all $\eps$. We can then define $\sigma_{\eps}:=\sup(K_{\eps})\in K_{\eps}$.
From $K\subseteq B_{M}(0)$, we get $\sigma:=[\sigma_{\eps}]\in K$.
It is not hard to prove that $\sigma=\sup(K)=\max(K)$. Analogously,
we can prove the existence of the sharp minimum of $K$.
\item If $S=(a,b)$, where $a$, $b\in\rcrho$ and $a\le b$, then $\sup(S)=b$
and $\inf(S)=a$.
\item If $S=\left\{ \frac{1}{n}\mid n\in\hypNr\right\} $, then $\inf(S)=0$.
\item Like in several other non-Archimedean rings, both sharp supremum and
infimum of the set $D_{\infty}$ of all infinitesimals do not exist.
In fact, by contradiction, if $\sigma$ were the sharp supremum of
$D_{\infty}$, then from \eqref{eq:2defSup} for $q=1$ we would get
the existence of $\bar{h}\in D_{\infty}$ such that $\sigma\leq\bar{h}+\diff{\rho}$.
But then $\sigma\in D_{\infty}$, so also $2\sigma\in D_{\infty}$.
Therefore, we get $2\sigma\le\sigma$ because $\sigma$ is an upper
bound of $D_{\infty}$, and hence $\sigma=0\ge\diff{\rho}$, a contradiction.
Similarly, one can prove that there does not exist the infimum of
this set.
\item Let $S=\left(0,1\right)_{\mathbb{R}}=\left\{ x\in\mathbb{R\,}|\,0<x<1\right\} $,
then clearly $\sigma=1$ is the Fermat supremum of $S$ whereas there
does not exist the sharp supremum of $S$. Indeed, if $\sigma=\sup(S)$,
then $s\le\sigma\le\bar{s}+\diff{\rho}$ for all $s\in S$ and for
some $\bar{s}\in S$. Taking any $s\in(\bar{s},1)_{\R}\subseteq S$
we get $s\le\sigma\le\bar{s}+\diff{\rho}$, which, for $\eps\to0$,
implies $s\le\bar{s}$ because $s$, $\bar{s}\in\R$. This contradicts
$s\in(\bar{s},1)$. In particular, $1$ is not the sharp supremum.
This example shows the importance of Def.~\ref{def:topSup}, i.e.~that
the best notion of supremum in a non-Archimedean setting depends on
a fixed topology.
\item \label{enu:01+dynamicPt}Let $S=(0,1)\cup\{\hat{s}\}$ where $\hat{s}|_{L}=2$,
$\hat{s}|_{L^{c}}=\frac{1}{2}$, $L\subzero I$, $L^{c}\subzero I$,
then $\nexists\,\sup(S)$. In fact, if $\exists\sigma:=\sup(S)$,
then $\sigma|_{L}\ge\hat{s}|_{L}=2$ and $\sigma|_{L^{c}}=1$. Assume
that $\exists\bar{s}\in S:\ \sigma-\diff{\rho}\le\bar{s}$, then $2-\diff{\rho}|_{L}\le\sigma|_{L}-\diff{\rho}|_{L}\le\bar{s}|_{L}$.
Thereby, $\bar{s}|_{L}>\frac{3}{2}$ and hence $\bar{s}\not\in(0,1)$
and $\bar{s}=\hat{s}$. We hence get $\sigma|_{L^{c}}-\diff{\rho}|_{L^{c}}\le\hat{s}|_{L^{c}}$,
i.e.~$1-\diff{\rho}|_{L^{c}}\le\frac{1}{2}$, which is impossible.
We can intuitively say that the subpoint $\hat{s}|_{L}$ creates a
``$\eps$-hole'' (i.e.~a ``hole'' only for some $\eps$) on the
right of $S$ and hence $S$ is not ``an $\eps$-continuum'' on
this side. Finally note that the point $u|_{L}:=2$ and $u|_{L^{c}}:=1$
is the least upper bound of $S$.
\end{enumerate}
\end{example}

\begin{lem}
\label{lem:propSupInf}Let $A$, $B\subseteq\rcrho$, then
\begin{enumerate}
\item \label{enu:homogPos}$\forall\lambda\in\rcrho_{>0}:\ \sup(\lambda A)=\lambda\sup(A)$,
in the sense that one supremum exists if and only if the other one
exists, and in that case they coincide;
\item \label{enu:homoNeg}$\forall\lambda\in\rcrho_{<0}:\ \sup(\lambda A)=\lambda\inf(A)$,
in the sense that one supremum/infimum exists if and only if the other
one exists, and in that case they coincide;
\end{enumerate}
Moreover, if $\exists\sup(A)$, $\sup(B)$, then:
\begin{enumerate}[resume]
\item \label{enu:incr}If $A\subseteq B$, then $\sup(A)\le\sup(B)$;
\item \label{enu:addit}$\sup(A+B)=\sup(A)+\sup(B)$;
\item \label{enu:prodSup}If $A$, $B\subseteq\rcrho_{\ge0}$, then $\sup(A\cdot B)=\sup(A)\cdot\sup(B)$.
\end{enumerate}
\end{lem}

\begin{proof}
\ref{enu:homogPos}: If $\exists\sup(\lambda A)$, then we have $a\le\frac{1}{\lambda}\sup(\lambda A)$
for all $a\in A$. For all $q\in\N$, we can find $\bar{a}\in A$
such that $\sup(\lambda A)-\lambda\bar{a}\le\diff{\rho}^{q}$. Thereby,
$\frac{1}{\lambda}\sup(\lambda A)-\bar{a}\le\frac{1}{\lambda}\diff{\rho}^{q}\to0$
as $q\to+\infty$ because $\lambda$ is moderate. This proves that
$\exists\sup(A)=\frac{1}{\lambda}\sup(\lambda A)$. Similarly, we
can prove the opposite implication.

\ref{enu:homoNeg}: From \ref{enu:homogPos} and \eqref{eq:infSup}
we get: $\sup(\lambda A)=\sup(-\lambda(-A))=-\lambda\sup(-A)=\lambda\inf(A)$.

\ref{enu:incr}: By contradiction, using Lem.~\ref{lem:negationsSubpoints}.\ref{enu:neg-le},
if $\sup(A)>_{L}\sup(B)$ for some $L\subzero I$, then $\sup(A)-\sup(B)>_{L}\diff{\rho}^{q}$
for some $q\in\N$ by Lem.~\ref{lem:mayer} for the ring $\rcrho|_{L}$.
Property \eqref{eq:2defSup} yields $\sup(A)-\diff{\rho}^{q}\le\bar{a}$
for some $\bar{a}\in A$, and $\bar{a}\le\sup(B)$ because $A\subseteq B$.
Thereby, $\sup(A)-\sup(B)\le\diff{\rho}^{q}$, which implies $\diff{\rho}^{q}<_{L}\diff{\rho}^{q}$,
a contradiction.

\ref{enu:addit} and \ref{enu:prodSup} follow easily from Def.~\ref{def:topSup}.\ref{enu:1defSup}
and \eqref{eq:2defSup}.
\end{proof}
In the next section, we introduce in the non-Archimedean framework
$\rcrho$ how to approximate $\sup(S)$ of $S\subseteq\rcrho$ using
points of $S$ and upper bounds, and the non-Archimedean analogous
of the notion of upper bound.

\subsection{\label{subsec:Archimedean-upper-bound}Approximations of $Sup$,
completeness from above and Archimedean upper bounds}

In the real field, we have the following peculiar properties:
\begin{enumerate}
\item \label{enu:LUPSup}The notion of least upper bound coincides with
that of close supremum, i.e.~it satisfies property \eqref{eq:closedSupR}.
We can hence question when these two notions coincide also in $\rcrho$.
Example \ref{exa:supInf}.\ref{enu:01+dynamicPt} shows that the answer
is not trivial. A first solution of this problem is already contained
in \cite[Prop.~1.4]{GarVer11}, where it is shown that the close supremum,
assuming that it exists, coincides with the least upper bound.
\item \label{enu:AUB}The notion of upper bound in $\R$ is very useful
because it entails the existence of the supremum. Clearly, since there
are infinite upper bounds but only one supremum, the notion of upper
bound results to be really useful in estimates with inequalities.
Moreover, in the ring $\rcrho$, the presence of infinite numbers
(of different magnitudes) allows one to have trivial upper bounds,
such as in the case $S=(0,1)$ and $M=\diff{\rho}^{-1}$, or $S=(0,\diff{\rho}^{-1})$
and $M=\diff{\rho}^{-2}$. Therefore, we can also investigate whether
we can consider non trivial upper bounds, i.e.~numbers which are,
intuitively, of the same order of magnitude of the elements of $S\subseteq\rcrho$.
On the other hand, example \ref{exa:supInf}.\ref{enu:01+dynamicPt}
shows that with respect to any reasonable definition of ``same order
of magnitude'', the upper bound $m=3$ must be of the same order
of any point in $S$, although $\nexists\,\sup(S)$. We will solve
this problem by introducing the definition of \emph{Archimedean upper
bound}.
\item \label{enu:seqSup}If $\emptyset\ne S\subseteq\R$ admits an upper
bound, then $\sup(S)$ can be arbitrarily approximated using upper
bounds and points of $S$. When is this possible if $\emptyset\ne S\subseteq\rcrho$?
\end{enumerate}
\noindent Example \ref{exa:supInf}.\ref{enu:01+dynamicPt} shows
that these problems cannot be solved in general, and we are hence
searching for a useful sufficient condition on $S$. As we will see
more clearly below, we could also say that we are searching for a
practical notion or procedure ``at the $\eps$-level'' (i.e.~working
on representatives) to determine whether a set has the supremum or
the least upper bound. However, we are actually far from a real solution
of this non trivial problem, and the present section presents only
preliminary steps in this direction.

We first prove the following useful characterization of the existence
of $\sup(S)$, which also solves problem \ref{enu:seqSup}:
\begin{thm}
\label{thm:hans}Let $S\subseteq\rcrho$, and let $U\subseteq\rcrho$
denote the set of upper bounds of $S$. Then $S$ has supremum if
and only if
\begin{equation}
\forall q\in\N\,\exists u_{q}\in U\,\exists s_{q}\in S:\ u_{q}-s_{q}\le\diff{\rho}^{q}.\label{eq:hans}
\end{equation}
\end{thm}

\begin{proof}
If $\sigma=\sup(S)$, then \eqref{eq:hans} simply follows by setting
$u_{q}:=\sigma$ and $s_{q}\in S$ from \eqref{eq:2defSup}. Vice
versa, if \eqref{eq:hans} holds, then 
\[
-\diff{\rho}^{q}\le s_{q}-u_{q}\le u_{q+1}-u_{q}\le u_{q+1}-s_{q+1}\le\diff{\rho}^{q+1}\quad\forall q\in\N.
\]
Thereby, $-(p-q)\diff{\rho}^{q}\le u_{p}-u_{q}\le(p-q)\diff{\rho}^{p}$
for all $p>q$, and hence $-\diff{\rho}^{\min(p,q)-1}\le u_{p}-u_{q}\le\diff{\rho}^{\min(p,q)-1}$
for all $p$, $q\in\N_{>0}$. This shows that $(u_{q})_{q\in\N}$
is a Cauchy sequence which thus converges to some $\sigma\in\rcrho$.
Property \eqref{eq:hans} yields that also $(s_{q})_{q\in\N}\to\sigma$,
and this implies condition \eqref{eq:2defSup}. Since each $u_{q}$
is an upper bound, for all $s\in S$ we have $s\le u_{q}$, which
gives $s\le\sigma$ for $q\to+\infty$.
\end{proof}
\noindent To solve problem \ref{enu:AUB}, assume that $u\in\rti$
is an upper bound of a non empty $S\subseteq\rti$. Let $[u_{\eps}]=u$,
and for all $s\in S$ choose a representative $[s_{\eps}(u)]=s$ such
that
\begin{equation}
\forall\eps\in I:\ s_{\eps}(u)\le u_{\eps}.\label{eq:reprForallEps}
\end{equation}
We first note that setting
\begin{equation}
\sigma_{\eps}:=\sup\left\{ s_{\eps}(u)\mid s\in S\right\} \quad\forall\eps\in I\label{eq:defSigmaSup}
\end{equation}
does not work to define a representative of the supremum, e.g.~if
$S=(0,1)$. Assume, e.g., that $u_{\eps}=3$ and take any sequence
$(s_{n})_{n\in\N}$ of different points of $S$: $s_{i}\ne s_{j}$
if $i\ne j$. Change representatives of $s_{n}=[s_{n\eps}]$ satisfying
\eqref{eq:reprForallEps} by setting $\bar{s}_{n\eps}:=s_{n\eps}(u)=s_{n\eps}(3)$
if $\eps\ne\frac{1}{n}$ and $\bar{s}_{n,\frac{1}{n}}:=3$. These
new representatives still satisfy \eqref{eq:reprForallEps}, but defining
$\sigma_{\eps}$ with them as in \eqref{eq:defSigmaSup}, we would
get $\sigma_{\frac{1}{n}}\ge\sup\left\{ \bar{s}_{n,\frac{1}{n}}\mid n\in\N_{>0}\right\} =3$,
and hence $[\sigma_{\eps}]\ne1=\sup(S)$. We want to refine this idea
by considering suitable representatives $\left[s_{\eps}(u)\right]=s$
satisfying \eqref{eq:reprForallEps}, and setting
\begin{align}
\sigma_{\eps}(S) & :=\sigma_{\eps}:=\inf\left\{ \sup\left\{ s_{\eps}(u)\mid s\in S\right\} \mid u\ge S\right\} \quad\forall\eps\in I,\label{eq:defSigma}\\
(\sigma_{\eps}(S)) & \in\R_{\rho}\ \Rightarrow\ \sigma(S):=\left[\sigma_{\eps}(S)\right]\in\rti,
\end{align}
where $u\ge S$ means that $u$ is an upper bound of $S$, and where
the representatives are chosen as follows: set $\R_{\infty}:=\R\cup\{+\infty\}$,
and for all $(u_{\eps})\in\R_{\infty}^{I}$ and $s\in S$:
\begin{equation}
\left\{ \begin{array}{l}
s\le[u_{\eps}]\in\rti\ \Rightarrow\ \exists\left[s_{\eps}(u)\right]=s\,\forall\eps\in I:\ s_{\eps}(u)\le u_{\eps}\\
(u_{\eps})\notin\R_{\rho}\text{ or }s\not\le[u_{\eps}]\ \Rightarrow\ \left[s_{\eps}(u)\right]=s\text{ is any representative of }s.
\end{array}\right.\label{eq:s_epsLeu_eps}
\end{equation}
Note that definition \eqref{eq:defSigma} depends on the chosen representatives
$(s_{\eps}(u))$ for $s\in S$ and $\left(u_{\eps}\right)$ for $u\ge S$;
trivially, if $\left(\bar{\sigma}_{\eps}(S)\right)$ is defined using
different representatives $(\bar{s}_{\eps}(u))$ and $(\bar{u}_{\eps})$,
and both $\left(\bar{\sigma}_{\eps}(S)\right)$ and $\left(\sigma_{\eps}(S)\right)$
well-define the supremum $\sup(S)$ (or the least upper bound $\text{lub}(S)$)
of $S$, then $\left[\bar{\sigma}_{\eps}(S)\right]=\left[\sigma_{\eps}(S)\right]$.
On the other hand, if we calculate $\left(\sigma_{\eps}(S)\right)$
using a certain choice of representatives, and we notice that $\left(\sigma_{\eps}(S)\right)$
is not an upper bound of $S$, we do not know whether another choice
of representatives can give an upper bound or not. This is one of
the weaknesses of the present solution. To highlight this dependence,
we will also sometimes use the following notations for our choice
functions (their existence depends on the axiom of choice):
\begin{align}
\text{e}(s,u,\eps) & :=s_{\eps}(u)\quad\forall s\in S\,\forall u\ge S\nonumber \\
\text{b}(u,\eps) & :=u_{\eps}\quad\forall u\ge S.\label{eq:choiceFncts}
\end{align}

We first observe that, for all $\eps\in I$:
\begin{align}
\nexists u & \ge S\ \Rightarrow\ \sigma_{\eps}=+\infty\nonumber \\
\exists u & \ge S\ \Rightarrow\ \sigma_{\eps}\le\sup\left\{ s_{\eps}(u)\mid s\in S\right\} \le u_{\eps}\label{eq:sigmaLEsupLEu}\\
S & =\emptyset\ \Rightarrow\ \sigma_{\eps}=\sup\left\{ s_{\eps}(u)\mid s\in S\right\} =-\infty.\nonumber 
\end{align}
We therefore have:
\begin{lem}
\label{lem:mainCondLUB}Assume that $S\subseteq\rti$, $(\sigma_{\eps}(S))\in\R_{\rho}$
and $\sigma(S)\ge S$. Then the following properties hold:
\begin{enumerate}
\item $\sigma(S)=\text{\emph{lub}}(S)$.
\item If $\text{\emph{b}}(\sigma(S),\eps)=\sigma_{\eps}(S)$, then $\sigma_{\eps}(S)=\sup\left\{ s_{\eps}(\sigma(S))\mid s\in S\right\} =\inf\left\{ u_{\eps}\mid u\ge S\right\} $
for all $\eps\in I$.
\end{enumerate}
\end{lem}

\begin{proof}
If $\sigma:=\sigma(S)\ge S$, inequality \eqref{eq:sigmaLEsupLEu}
shows that $\sigma$ is the least upper bound of $S.$ From \eqref{eq:defSigma}
and \eqref{eq:sigmaLEsupLEu}, we have $\sigma_{\eps}\le\sup\left\{ s_{\eps}(\sigma)\mid s\in S\right\} \le\sigma_{\eps}$
because $\sigma\ge S$ and $\text{b}(\sigma,\eps)=\sigma_{\eps}$
(i.e.~the chosen representative $\left(u_{\eps}\right)$ for the
upper bound $\sigma\ge S$ is exactly $(\sigma_{\eps})$ as defined
in \eqref{eq:defSigma}). Finally, the inequality $\sigma_{\eps}\le\inf\left\{ u_{\eps}\mid u\ge S\right\} $
follows from \eqref{eq:sigmaLEsupLEu}. The other inequality follows
from $\sigma=\sigma(S)\ge S$ and from $\text{b}(\sigma,\eps)=\sigma_{\eps}$.
\end{proof}
\noindent In general, the net $(\sigma_{\eps}(S))$ is not $\rho$-moderate.
In fact, if $(u_{n})_{n\in\N}$ is a sequence of different upper bounds
and we set $s_{n,\frac{1}{n}}(u_{n})=-\rho_{\frac{1}{n}}^{-1/n}$,
this yields $\sigma_{\frac{1}{n}}\le-\rho_{\frac{1}{n}}^{-1/n}$.
On the other hand, we have:
\begin{lem}
\label{lem:boundedImpliesMod}Let $u\in\rti$, $S\subseteq\rti$ with
$S\le u$. Assume that for some $\bar{s}\in S$ we have
\begin{equation}
\forall^{0}\eps:\ \sigma_{\eps}(S)\ge\bar{s}_{\eps}(u).\label{eq:aboveOne}
\end{equation}
Then $(\sigma_{\eps}(S))\in\R_{\rho}$ and $\bar{s}\le\sigma(S)\le u$.
\end{lem}

\begin{proof}
From \eqref{eq:sigmaLEsupLEu}, we get $\sigma_{\eps}\le u_{\eps}$.
The conclusion thus follows from \eqref{eq:aboveOne} and $\bar{s}$,
$u\in\rti$.
\end{proof}
\noindent Since the set of all infinitesimals $S=D_{\infty}$ has
no least upper bound, the previous two results imply that $\sigma(D_{\infty})\not\ge D_{\infty}$.
Using Lem.~\ref{lem:boundedImpliesMod} with $l=-r$, $u=r\in\R_{>0}$,
we have that $\sigma(D_{\infty})$ is always an infinitesimal (that
actually depends on the chosen representatives $\left(s_{\eps}(u)\right)$
and $(u_{\eps})$).

The following condition solves problem \ref{enu:AUB}:
\begin{defn}
\label{def:complAbove}Let $S\subseteq\rcrho$ and for simplicity
use $\sigma_{\eps}=\sigma_{\eps}(S)$, then we say that $S$ \emph{is
complete from above }if the following conditions hold:
\begin{enumerate}
\item \label{enu:sigmaUB}$\forall s\in S\,\exists[s_{\eps}]=s\,\forall^{0}\eps:\ s_{\eps}\le\sigma_{\eps}$.
\item \label{enu:CFA-diag}If $(s^{e})_{e\in I}$ is a family of $S$ which
satisfies:
\begin{equation}
\exists[u_{\eps}]\in\rti\,\forall e\in I\,\forall^{0}\eps:\ s_{\eps}^{e}(\sigma)\le u_{\eps}\label{eq:boundFam}
\end{equation}
then
\begin{equation}
\exists[\bar{s}_{\eps}]\in\overline{S}\,\forall^{0}\eps:\ s_{\eps}^{\eps}(\sigma)\le\bar{s}_{\eps},\label{eq:complAbove}
\end{equation}
where $\overline{S}$ is the closure of $S$ in the sharp topology.
\end{enumerate}
Moreover, if $\exists s\in S:\ s>0$, then we say that $M$ \emph{is
an Archimedean upper bound (AUB) of $S$} if
\begin{enumerate}[label=(\alph*)]
\item $M\in\rcrho$ and $\forall s\in S:\ s\le M$;
\item $\exists n\in\N\,\exists\bar{s}\in S:\ M<n\bar{s}$. The minimum $n\in\N$
that satisfies this property is called the \emph{order of }$M$ (clearly,
$n\ge2$). Note that this condition, using an Archimedean-like property,
formalizes the idea that $M$ and $\bar{s}$ are of the same order
of magnitude.
\end{enumerate}
Dually, we can define the notion of \emph{completeness from below}
by reverting all the inequalities in \ref{enu:sigmaUB} and \ref{enu:CFA-diag}.
If $\exists s\in S:\ s<0$, then $N$ is an \emph{Archimedean lower
bound (ALB) of $S$} if it is a lower bound such that $\exists n\in\N\,\exists\bar{s}\in S:\ \bar{s}n<N$.
\end{defn}

\noindent Note that $\sigma=\sup(S)$ is always an AUB of order 2.
In fact, from the existence of $s\in S_{>0}$, we have $s>\diff{\rho}^{q}$
for some $q\in\N$ and the existence of $\bar{s}\in S$ with $\bar{s}\ge\sigma-\diff{\rho}^{q+1}$.
Thereby, $\bar{s}\ge s-\diff{\rho}^{q+1}>\diff{\rho}^{q}-\diff{\rho}^{q+1}>\diff{\rho}^{q+1}$
and thus $\sigma\le\bar{s}+\diff{\rho}^{q+1}<2\bar{s}$. We also note
that $S=\rti$ is trivially complete from above (because $\sigma_{\eps}=+\infty$
from \eqref{eq:sigmaLEsupLEu}, and by setting $\bar{s}_{\eps}=u_{\eps})$
but $\nexists\sup(\rti)$. Looking at Lem.~\ref{lem:boundedImpliesMod},
in the case of a non empty subset $S\subseteq\rti$ bounded from above,
the condition of being complete from above can be intuitively described
as follows:
\begin{enumerate}[label=(\alph*)]
\item \label{enu:Choose-representatives}Choose representatives $[u_{\eps}]=u$
for each $u\ge S$ and $[s_{\eps}(u)]=s$ for each $s\in S$ satisfying
\eqref{eq:s_epsLeu_eps};
\item Define $\sigma_{\eps}(S)=:\sigma_{\eps}\in\R_{\infty}=\R\cup\{+\infty\}$
as in \ref{eq:defSigma}.
\item Check if the inequality $s_{\eps}(\sigma)\le\sigma_{\eps}$ holds
(in this case, for the chosen representatives satisfying \eqref{eq:s_epsLeu_eps},
without loss of generality, we can assume that $\text{b}(\sigma,\eps)=\sigma_{\eps}$
for all $\eps\in I$);
\item From \emph{any} family $(s^{e})_{e\in I}$ of $S$ (which is therefore
bounded from above, so that \eqref{eq:boundFam} always holds) pick
the diagonal net $(s_{\eps}^{\eps}(\sigma))$ from its representatives
(depending on $\sigma\ge S$) and check if $s_{\eps}^{\eps}(\sigma)\le\bar{s}_{\eps}$
for some $\bar{s}$ in the sharp closure $\overline{S}$.
\item If any of the two previous steps do not hold, consider a different
set of representatives in the first step \ref{enu:Choose-representatives}.
\end{enumerate}
We therefore have the following simplified case:
\begin{lem}
\label{lem:CFAboundedSets}Assume that $\emptyset\ne S\subseteq\rcrho$
is sharply bounded from above, then $S$ is complete from above if
and only if the following condition holds
\begin{enumerate}
\item $\sigma(S)=:\sigma\ge S$
\item If $(s^{e})_{e\in I}$ is a family of $S$, then $\exists[\bar{s}_{\eps}]\in\overline{S}\,\forall^{0}\eps:\ s_{\eps}^{\eps}(\sigma)\le\bar{s}_{\eps}$.
\end{enumerate}
\end{lem}

\noindent Note that example \ref{exa:supInf}.\ref{enu:01+dynamicPt}
satisfies the first one of these conditions (so that $\sigma(S)$
is its least upper bound) but not the second one because it does not
admit supremum (see the following theorem). Cases which remain excluded
from the previous lemma are e.g.~intervals $(a,+\infty)$, with $-\infty\le a\in\rti$
which are complete from above even if they do not admit supremum nor
least upper bound.

\noindent The following results solve the remaining problems \ref{enu:LUPSup}
and \ref{enu:AUB} we set at the beginning of this section.
\begin{thm}
\label{thm:seqSup}Assume that $\emptyset\ne S\subseteq\rcrho$, then
\begin{enumerate}
\item \label{enu:SupExists}If $S$ is complete and bounded from above and
$\text{\emph{b}}(\sigma(S),\eps)=\sigma_{\eps}(S)$, then $\exists\sup(S)=\sigma(S)$;
\end{enumerate}
Let $(s_{q})_{q\in\N}$ and $(u_{q})_{q\in\N}$ be two sequences as
in Thm.~\ref{thm:hans}, then
\begin{enumerate}[resume]
\item \label{enu:seqSupAUB}If $\exists s\in S:\ s>0$, and if there exists
$C\in\R_{>0}$ such that $s_{q}\ge C\diff{\rho}^{q}$ for all $q\in\N$
large, then $u_{q}$ is an AUB of $S$ for all $q$ sufficiently large;
\item \label{enu:seqSupAUB2}If $\exists s\in S:\ s>0$, then $u_{q}$ is
an AUB of $S$ of order $2$ for all $q$ sufficiently large.
\end{enumerate}
\begin{proof}
\ref{enu:SupExists}: From Lem.~\ref{lem:boundedImpliesMod} we get
that $\sigma(S)=:\sigma$ is well-defined because $\sigma\ge S$ by
definition of completeness from above, i.e.~Def.~\ref{def:complAbove}.\ref{enu:sigmaUB}.
Therefore, Lem.~\ref{lem:mainCondLUB} and the assumption $\text{b}(\sigma(S),\eps)=\sigma_{\eps}(S)$,
yield that $\sigma_{\eps}=\sup\left\{ s_{\eps}(\sigma)\mid s\in S\right\} $
for all $\eps$. For arbitrary $q\in\N$ and $e\in I$, this yields
\begin{equation}
\sigma_{e}-\rho_{e}^{q+2}<s_{e}^{e}(\sigma)=:s_{e}^{e}\label{eq:sqe}
\end{equation}
for some $s^{e}\in S$ (that depends on both $q$ and $e$). By definition
of completeness from above, we get the existence of $\bar{s}=[\bar{s}_{\eps}]\in\overline{S}$
such that $s_{\eps}^{\eps}\le\bar{s}_{\eps}$ for $\eps$ mall. Setting
$e=\eps$ in \eqref{eq:sqe}, we get $\sigma_{\eps}-\rho_{\eps}^{q+2}<s_{\eps}^{\eps}\le\bar{s}_{\eps}$
for $\eps$ small, i.e.~$\sigma-\diff{\rho}^{q+2}<\bar{s}$. Since
$\bar{s}\in\overline{S}$, there exists $s\in S\cap(\bar{s}-\diff{\rho}^{q+1},\bar{s}+\diff{\rho}^{q+1})$.
Thereby, $\sigma-\diff{\rho}^{q}+\diff{\rho}^{q+1}<\sigma-\diff{\rho}^{q+2}<\bar{s}$,
and hence $\sigma-\diff{\rho}^{q}<\bar{s}-\diff{\rho}^{q+1}<s$, which
proves our claim \ref{enu:SupExists}.

Now, assume that $s_{q}\ge C\diff{\rho}^{q}$ for some $C\in\R_{>0}$
and for all $q\in\N$ sufficiently large. Then, for these $q$ we
have $\frac{s_{q}+\diff{\rho}^{q}}{s_{q}}\le1+\frac{1}{C}\le\left\lceil 1+\frac{1}{C}\right\rceil =:n\in\N$.
This yields $u_{q}<s_{q}+\diff{\rho}^{q}<ns_{q}$, i.e.~$u_{q}$
is an AUB of $S$. Finally, from the existence of at least one $s\in S_{>0}$,
we get the existence of $p\in\N$ such that $s>\diff{\rho}^{p}$.
Therefore, also $\diff{\rho}^{p}<s\le\sigma$. From \ref{enu:SupExists},
we hence get that for $q\in\N$ sufficiently large $\diff{\rho}^{p}<s_{q}\le\sigma$,
i.e.~$\frac{1}{s_{q}}<\diff{\rho}^{-p}$ and $\frac{s_{q}+\diff{\rho}^{q}}{s_{q}}\le1+\diff{\rho}^{q-p}\le2$
for all $q>p$. Proceeding as above we can prove the claim.
\end{proof}
\end{thm}

\noindent Example \ref{exa:supInf}.\ref{enu:01+dynamicPt} shows
the necessity of the assumption of completeness from above in this
theorem.

Directly from Thm.~\ref{thm:seqSup}.\ref{enu:SupExists}, we obtain:
\begin{cor}
\label{cor:existenceSupCFA}Let $\emptyset\ne S\subseteq\rcrho$.
Assume that $S$ is complete from above and $\text{\emph{b}}(\sigma(S),\eps)=\sigma_{\eps}(S)$,
then $\exists\sup(S)$ if and only if $S$ admits an upper bound.
\end{cor}

Now, we can also complete the relationships between close supremum
and least upper bound (see also \cite[Prop.~1.4]{GarVer11}) and study
what happens if we consider only the upper bounds $u$ lower than
a fixed upper bound $\bar{u}$ in \eqref{eq:defSigma}.
\begin{cor}
\label{cor:supLUB}Let $\emptyset\ne S\subseteq\rcrho$, then the
following properties hold:
\begin{enumerate}
\item \label{enu:supLUB}If $\exists\sup(S)=\sigma$, then $\exists\text{\emph{lub}}(S)=\sigma$.
\item \label{enu:sup=00003DLUB}If $S$ is complete and bounded from above,
then
\[
\exists\sup(S)=\sigma\ \iff\ \exists\,\text{\emph{lub}}(S)=\sigma.
\]
\item \label{enu:lessUB}Assume that $\bar{u}\ge S$ and define $\bar{\sigma}_{\eps}(S):=\inf\left\{ \sup\left\{ s_{\eps}(u)\mid s\in S\right\} \mid\bar{u}\ge u\ge S\right\} $.
Then $\bar{\sigma}(S):=\left[\bar{\sigma}_{\eps}(S)\right]$ is well-defined
and $\bar{\sigma}(S)\ge\sigma(S)$. If $\sigma(S)\ge S$, then $\bar{\sigma}(S)=\sigma(S)$.
If $\bar{\sigma}(S)\ge S$, then $\bar{\sigma}(S)$ is the least upper
bound of $S$, thus $\bar{\sigma}(S)=\sigma(S)$ if $S$ is complete
from above.
\item \label{enu:SupImpliesCFA}Assume that $\exists\,\sigma(S)\ge S$,
$\text{\emph{b}}(\sigma(S),\eps)=\sigma_{\eps}(S)$ and $\exists\sup(S)$.
Then $S$ is complete from above.
\end{enumerate}
\end{cor}

\begin{proof}
\ref{enu:supLUB} and \ref{enu:sup=00003DLUB}: Assume that $\exists\sup(S)=\sigma$,
and let $u$ be an upper bound of $S$; by condition \eqref{eq:2defSup}
we get $\sigma-\diff{\rho}^{q}\le s_{q}\le u$ for all $q\in\N$ and
for some $s_{q}\in S$. For $q\to+\infty$, we get $\sigma\le u$.
Vice versa, if $S\ne\emptyset$ is complete from above and $\sigma$
is the least upper bound of $S$, then the conclusion follows from
Cor.~\ref{cor:existenceSupCFA}.

\ref{enu:lessUB}: If $\bar{s}\in S\le u$, we can prove that $(\bar{\sigma}_{\eps}(S))\in\R_{\rho}$
and $\bar{s}\le\bar{\sigma}(S)\le u$ as in the proof of Lem.~\ref{lem:boundedImpliesMod}.
We always have that $\sigma(S)\le\bar{u}$ because $\bar{u}\ge S$.
Therefore, if $\sigma(S)\ge S$, then $\bar{u}\ge\sigma(S)\ge S$
and hence $\bar{\sigma}(S)\le\sigma(S)\le\bar{\sigma}(S)$. Finally,
if we assume that $\bar{\sigma}(S)\ge S$ and we consider an arbitrary
upper bound $u\ge S$, then either $u\ge\bar{u}$ or $u<_{L}\bar{u}$
for some $L\subzero I$. Thereby, $\bar{\sigma}(S)\le u$ or $\bar{\sigma}(S)\le_{L}u$,
and hence $\bar{\sigma}(S)\le u$. Therefore, $\bar{\sigma}(S)$ is
the least upper bound of $S$, and the final claim follows from \ref{enu:sup=00003DLUB}.

\ref{enu:SupImpliesCFA}: From Lem.~\ref{lem:mainCondLUB}, we have
$\sigma(S)=\text{lub}(S)=:\sigma$ and hence $\sup(S)=\sigma\in\overline{S}$
from \ref{enu:LUPSup}. From $\text{\emph{b}}(\sigma(S),\eps)=\sigma_{\eps}(S)$
and \eqref{eq:s_epsLeu_eps} we have $s_{\eps}(\sigma)\le\sigma_{\eps}$
for all $s\in S$ and all $\eps\in I$. In particular, if $\left(s^{e}\right)_{e\in I}$
is a family of $S$, we have $s_{\eps}^{e}(\sigma)\le\sigma_{\eps}$
for all $e\in I$ and all $\eps\in I$. Taking $e=\eps$, we get that
Def.~\ref{def:complAbove}.\ref{enu:CFA-diag} holds.
\end{proof}
\begin{example}
\label{exa:CFA}~
\begin{enumerate}
\item Example \ref{exa:supInf}.\ref{enu:01+dynamicPt} shows that the assumption
of being complete from above is necessary in Cor.~\ref{cor:supLUB}.
On the other hand, using the notation of this example, one can prove
that $\sigma_{\eps}(S)=2$ if $\eps\in L$ and $\sigma_{\eps}(S)=1$
if $\eps\in L^{c}$. From Lem.~\ref{lem:mainCondLUB} it follows
that $\sigma(S)$ is the least upper bound of $S$. This underscores
the differences between the order theoretical definition of supremum
as least upper bound and the topological definition of closed supremum.
\item Any set having a maximum is trivially complete from above: set $[\bar{s}_{\eps}]:=\max(S)$
in \eqref{eq:complAbove} and consider that $\sigma(S)=\max(S)$.
\item $S=(0,1)$ is complete from above for $[\bar{s}_{\eps}]=1$ and because
$\sigma_{\eps}(S)=:\sigma_{\eps}=1$. In fact, $\sigma_{\eps}\le1$
from \eqref{eq:sigmaLEsupLEu}. Now, take any $u=[u_{\eps}]\ge S$,
so that $u_{\eps}\ge1$ for all $\eps\ge\eps_{0}$. For $\eps\ge\eps_{0}$,
by contradiction assume that $1>\sup\left\{ s_{\eps}(u)\mid s\in S\right\} $,
and hence $1>r>\sup\left\{ s_{\eps}(u)\mid s\in S\right\} $ for some
$r\in(0,1)_{\R}\subseteq S$. Take $r_{\eps}=r$ as representative
of $r$ in \eqref{eq:s_epsLeu_eps}; we have two cases: If $r_{\eps}(u)=u_{\eps}\ge1$,
then $1>\sup\left\{ s_{\eps}(u)\mid s\in S\right\} \ge r_{\eps}(u)=u_{\eps}\ge1$;
if $r_{\eps}(u)=r$, then $\sup\left\{ s_{\eps}(u)\mid s\in S\right\} \ge r_{\eps}(u)=r>\sup\left\{ s_{\eps}(u)\mid s\in S\right\} $.
In any case, we get a contradiction, and this proves that $1\le\sup\left\{ s_{\eps}(u)\mid s\in S\right\} $
for all $\eps\le\eps_{0}$, and hence $\sigma_{\eps}\ge1$.
\item There do not exist neither the supremum nor the least upper bound
of $S=1+D_{\infty}$. On the other hand, $2$ is an AUB of $S$ and
hence $S$ is not complete from above.
\item $D_{\infty}$ has neither AUB nor ALB; $\rcrho$ has neither AUB nor
ALB; $\{\diff{\rho}^{r}\mid r\in\R_{>0}\}$ has no supremum and no
AUB and hence it is not complete from above.
\item Assume that there does not exist and upper bound of $S$. This means
that
\[
\forall u\in\rcrho\,\exists s\in S:\ s\sbpt{>}u.
\]
Thereby, there exists a sequence $(s_{q})_{q\in\N}$ of $S$ such
that $s_{q}\sbpt{>}\diff{\rho}^{-q}$. Based on this, we could set
$\sup(S):=+\infty$.
\end{enumerate}
\end{example}

\section{The hyperlimit of a hypersequence}

\subsection{Definition and examples}
\begin{defn}
\label{def:hyperlimit}A map $x:\hyperN{\sigma}\ra\rcrho$, whose
domain is the set of hypernatural numbers $\hyperN{\sigma}$ is called
a ($\sigma-$) \emph{hypersequence} (of elements of $\rcrho$). The
values $x(n)\in$ $\rcrho$ at $n\in$$\hyperN{\sigma}$ of the function
$x$ are called \emph{terms} of the hypersequence and, as usual, denoted
using an index as argument: $x_{n}=x(n)$. The hypersequence itself
is denoted by $(x_{n})_{n\in\hyperN{\sigma}}$, or simply $(x_{n})_{n}$
if the gauge on the domain is clear from the context. Let $\sigma$,
$\rho$ be two gauges, $x:\hyperN{\sigma}\ra\rcrho$ be a hypersequence
and $l\in\rcrho$. We say that $l$ is \emph{hyperlimit} of $(x_{n})_{n}$
as $n$$\rightarrow\infty$ and $n$$\in\hyperN{\sigma}$, if
\[
\forall q\in N\,\exists M\in\hyperN{\sigma}\,\forall n\in\hyperN{\sigma}_{\geq M}:\ |x_{n}-l|<\diff{\rho}^{q}.
\]
In the following, if not differently stated, $\rho$ and $\sigma$
will always denote two gauges and $(x_{n})_{n}$ a $\sigma$-hypersequence
of elements of $\rcrho$. Finally, if $\sigma_{\eps}\ge\rho_{\eps}$,
at least for all $\eps$ small, we simply write $\sigma\ge\rho$.
\end{defn}

\begin{rem}
In the assumption of Def.~\ref{def:hyperlimit}, let $k\in\rcrho_{>0}$,
$N\in\N$, then the following are equivalent:
\begin{enumerate}
\item $l$$\in\rcrho$ is the hyperlimit of $(x_{n})_{n}$ as $n\in\hyperN{\sigma}$.
\item $\forall\eta\in\rcrho_{>0}\,\exists M\in\hyperN{\sigma}\,\forall n\in\hyperN{\sigma}_{\geq M}:\ |x_{n}-l|<\eta$.
\item Let $U\subseteq\rcrho$ be a sharply open set, if $l\in U$ then $\exists M\in\hyperN{\sigma}\,\forall n\in\hyperN{\sigma}_{\geq M}:\ x_{n}\in U$.
\item $\forall q\in\N\,\exists M\in\hyperN{\sigma}\,\forall n\in\hyperN{\sigma}_{\geq M}:\ |x_{n}-l|<k\cdot\diff{\rho}^{q}$.
\item $\forall q\in\N\,\exists M\in\hyperN{\sigma}\,\forall n\in\hyperN{\sigma}_{\geq M}:\ |x_{n}-l|<\diff{\rho}^{q-N}$.
\end{enumerate}
\end{rem}

\noindent Directly by the inequality $|l_{1}-l_{2}|\leq|l_{1}-x_{n}|+|l_{2}-x_{n}|\leq2\diff{\rho}^{q+1}<\diff{\rho}^{q}$
(or by using that the sharp topology on $\rcrho$ is Hausdorff) it
follows that there exists at most one hyperlimit, so that we can use
the notation

\begin{eqnarray*}
\hyperlim{\rho}{\sigma}x_{n} & := & l.
\end{eqnarray*}

\noindent As usual, a hypersequence (not) having a hyperlimit is said
to be (non-)convergent. We can also similarly say that $(x_{n})_{n}:\hyperN{\sigma}\ra\rcrho$
is divergent to $+\infty$ ($-\infty$) if
\[
\forall q\in\mathbb{N}\,\exists M\in\hyperN{\sigma}\,\forall n\in\hyperN{\sigma}_{\geq M}:\ x_{n}>\diff{\rho}^{-q}\quad(x<-\diff{\rho}^{-q}).
\]

\begin{example}
\label{exa:hyperlimits}\ 
\begin{enumerate}
\item If $\sigma\le\rho^{R}$ for some $R\in\R_{>0}$, we have $\hyperlim{\rho}{\sigma}\frac{1}{n}=0$.
In fact, $\frac{1}{n}<\diff{\rho}^{q}$ holds e.g.~if $n>\left[\text{int}\left(\rho_{\eps}^{-q}\right)+1\right]\in\hyperN{\sigma}$
because $\rho_{\eps}^{-q}\le\sigma_{\eps}^{-q/R}$ for $\eps$ small.
\item Let $\rho$ be a gauge and set $\sigma_{\eps}:=\exp\bigg(-\rho_{\eps}^{-\frac{1}{\rho_{\eps}}}\bigg)$,
so that $\sigma$ is also a gauge. We have 
\[
\hyperlim{\rho}{\sigma}\frac{1}{\log n}=0\in\rcrho\text{ whereas \ensuremath{\nexists}}\hyperlim{\rho}{\rho}\frac{1}{\log n}
\]
In fact, if $n>1$, we have $0<\frac{1}{\log n}<\diff\rho{}^{q}$
if and only if $\log n>\diff\rho{}^{-q}$, i.e. $n>e^{\diff\rho{}^{-q}}$(in
$\rcrho$). We can thus take $M:=\bigg[\text{int}\bigg(e^{\rho_{\eps}^{-q}}\bigg)+1\bigg]\in\hyperN{\sigma}$
because $e^{\rho_{\eps}^{-q}}<\exp\bigg(\rho_{\eps}^{-\frac{1}{\rho_{\eps}}}\bigg)=\sigma_{\eps}^{-1}$
for $\eps$ small. Vice versa, by contradiction, if $\exists\hyperlim{\rho}{\rho}\frac{1}{\log n}=:l\in\rcrho$,
then by the definition of hyperlimit from $\hyperN{\rho}$ to $\rcrho$,
we would get the existence of $M\in\hyperN{\rho}$ such that
\begin{equation}
\forall n\in\hyperN{\rho}:n\ge M\ \Rightarrow\ \frac{1}{\log n}-\diff{\rho}<l<\frac{1}{\log n}+\diff{\rho}\label{eq:hyperlimit2-1}
\end{equation}
We have to explore two possibilities: if $l$ is not invertible, then
$l_{\eps_{k}}=0$ for some sequence $(\eps_{k})\downarrow0$ and some
representative $[l_{\eps}]=l$. Therefore from \ref{def:hyperlimit},
we get 
\[
\frac{1}{\log M_{\eps_{k}}}<l_{\eps_{k}}+\rho_{\eps_{k}}=\rho_{\eps_{k}}
\]
hence $M_{\eps_{k}}>e^{-\frac{1}{\rho_{_{\eps_{k}}}}}$ $\forall k\in\mathbb{N}$,
in contradiction with $M\in\rcrho$. If $l$ is invertible, then $\diff\rho^{p}<|l|$
for some $p\in\mathbb{N}$. Setting $q:=\min\{p\in\mathbb{N}\mid\diff\rho^{p}<|l|\}+1$,
we get that $l_{\bar{\eps}_{k}}<\rho_{\bar{\eps}_{k}}^{q}$ for some
sequence $(\bar{\eps}_{k})_{k}\downarrow0$. Therefore 
\[
\frac{1}{\log M_{\bar{\eps}_{k}}}<l_{\bar{\eps}_{k}}+\rho_{\bar{\eps}_{k}}\le|l_{\bar{\eps}_{k}}|+\rho_{\bar{\eps}_{k}}<\rho_{\bar{\eps}_{k}}^{q}+\rho_{\bar{\eps}_{k}}
\]
and hence $M_{\bar{\eps}_{k}}>\exp\bigg(\frac{1}{\rho_{\bar{\eps}_{k}}^{q}+\rho_{\eps_{k}}}\bigg)$
for all $k\in\mathbb{N}$, which is in contradiction with $M\in\rcrho$
because $q\ge1$.\\
Analogously, we can prove that $\hyperlim{\rho}{\sigma}\frac{1}{\log(\log n)}=0$
if $\sigma=[\sigma_{\epsilon}]=\left[e^{-e^{\rho_{\epsilon}^{-\frac{1}{\rho_{\epsilon}}}}}\right]$
whereas $\nexists\,\hyperlim{\rho}{\rho}\frac{1}{\log(\log n)}$ (and
similarly using $\log(\log(\ptind^{k}(\log n)\ldots)$.
\item Set $x_{n}:=\diff{\rho}^{-n}$ if $n\in\N$, and $x_{n}:=\frac{1}{n}$
if $n\in\hyperN{\rho}\setminus\N$, then $\{x_{n}\mid n\in\hyperN{\rho}\}$
is unbounded in $\rcrho$ even if $\hyperlim{\rho}{\rho}x_{n}=0$.
Similarly, if $x_{n}:=\diff{\rho}^{n}$ if $n\in\N$ and $x_{n}:=\sin(n)$
otherwise, then $\lim_{\substack{n\to+\infty\\
n\in\N
}
}x_{n}=0$ whereas $\nexists\,\hyperlim{\rho}{\rho}x_{n}$. In general, we can
hence only state that convergent hypersequence are eventually bounded:
\[
\exists\,\hyperlim{\rho}{\sigma}x_{n}\ \Rightarrow\ \exists M\in\rcrho\,\exists N\in\hypNs\,\forall n\in\hypNs_{\ge N}:\ |x_{n}|\le M.
\]
\item If $k\sbpt{<}1$ and $k\sbpt{>}1$, then $\hyperlim{\rho}{\rho}k^{n}\sbpt{=}0$
and $\hyperlim{\rho}{\rho}k^{n}\sbpt{=}+\infty$, hence $\nexists\hyperlim{\rho}{\rho}k^{n}$.
\item Since for $n\in\N$ we have $(1-\diff{\rho})^{n}=1-n\diff{\rho}+O_{n}(\diff{\rho}^{2})$,
it is not hard to prove that $\left((1-\diff{\rho})^{n}\right)_{n\in\N}$
is not a Cauchy sequence. Therefore, $\nexists\lim_{n\in\N}(1-\diff{\rho})^{n}$,
whereas $\hyperlim{\rho}{\rho}(1-\diff{\rho})^{n}=0$.
\end{enumerate}
\end{example}

A sufficient condition to extend an ordinary sequence $(a_{n})_{n\in\N}:\N\ra\rcrho$
of $\rho$-generalized numbers to the whole $\hyperN{\sigma}$ is
\begin{equation}
\forall n\in\hyperN{\sigma}:\ \left(a_{\nint(n)_{\eps}}\right)\in\R_{\rho}.\label{eq:ext_a}
\end{equation}
In fact, in this way $a_{n}:=\left[a_{\nint(n)_{\eps}}\right]\in\rcrho$
for all $n\in\hypNs$, is well-defined because of Lem.~\ref{lem:nearestInt};
on the other hand, we have defined an extension of the old sequence
$(a_{n})_{n\in\N}$ because if $n\in\N$, then $\nint(n)_{\eps}=n$
for $\eps$ small and hence $a_{n}=[a_{n}]$. For example, the sequence
of infinities $a_{n}=\frac{1}{n}+\diff{\rho}^{-1}$ for all $n\in\N$
can be extended to any $\hyperN{\sigma}$, whereas $a_{n}=\diff{\sigma}^{-n}$
can be extended as $a:\hyperN{\sigma}\ra\rcrho$ only for some gauges
$\rho$, e.g.~if the gauges satisfy
\begin{equation}
\exists N\in\N\,\forall n\in\N\,\forall^{0}\eps:\ \sigma_{\eps}^{n}\ge\rho_{\eps}^{N},\label{eq:extendingPowers}
\end{equation}
(e.g.~$\sigma_{\eps}=\eps$ and $\rho_{\eps}=\eps^{1/\eps}$).

The following result allows us to obtain hyperlimits by proceeding
$\eps$-wise
\begin{thm}
\label{thm:epsWiseHyperlim}Let $(a_{n,\eps})_{n,\eps}:\N\times I\ra\R$.
Assume that for all $\eps$ 
\begin{equation}
\exists\lim_{n\to+\infty}a_{n,\eps}=:l_{\eps},\label{eq:epsLimAssumption}
\end{equation}
and that $l:=[l_{\eps}]\in\rcrho$. Then there exists a gauge $\sigma$
(not necessarily a monotonic one) such that
\begin{enumerate}
\item There exists $M\in\hyperN{\sigma}$ and a hypersequence $(a_{n})_{n}:\hyperN{\sigma}\ra\rcrho$
such that $a_{n}=[a_{\nint{(n)}_{\eps},\eps}]\in\rcrho$ for all $n\in\hyperN{\sigma}_{\ge M}$;
\item $l=\hyperlim{\rho}{\sigma}a_{n}$.
\end{enumerate}
\end{thm}

\begin{proof}
From \eqref{eq:epsLimAssumption}, we have
\begin{equation}
\forall\eps\,\forall q\,\exists M_{\eps q}\in\N_{>0}\,\forall n\ge M_{\eps q}:\ \rho_{\eps}^{q}-l_{\eps}<a_{n,\eps}<\rho_{\eps}^{q}+l_{\eps}.\label{eq:epsLim}
\end{equation}
Without loss of generality, we can assume to have recursively chosen
$M_{\eps q}$ so that
\begin{equation}
M_{\eps q}\le M_{\eps,q+1}\quad\forall\eps\,\forall q.\label{eq:M_eps_q-ineq}
\end{equation}
Set $\bar{M}_{\eps}:=M_{\eps,\lceil\frac{1}{\eps}\rceil}>0$; since
$\forall q\in\N\,\forall^{0}\eps:\ q\le\lceil\frac{1}{\eps}\rceil$,
\eqref{eq:M_eps_q-ineq} implies
\begin{equation}
\forall q\in\N\,\forall^{0}\eps:\ \bar{M}_{\eps}\ge M_{\eps q}.\label{eq:Mbar-M_q-rel}
\end{equation}
If the net $(\bar{M}_{\eps})$ is $\rho$-moderate, set $\sigma:=\rho$,
otherwise set $\sigma_{\eps}:=\min\left(\rho_{\eps},\bar{M}_{\eps}^{-1}\right)\in(0,1]$.
Thereby, the net $\sigma_{\eps}\to0$ as $\eps\to0^{+}$ (note that
not necessarily $\sigma$ is non-decreasing, e.g.~if $\lim_{\eps\to\frac{1}{k}}\bar{M}_{\eps}=+\infty$
for all $k\in\N_{>0}$ and $\bar{M}_{\eps}\ge\rho_{\eps}^{-1}$),
i.e.~it is a gauge. Now set $\bar{M}:=[\bar{M}_{\eps}]\in\hypNs$
because our definition of $\sigma$ yields $\bar{M}_{\eps}\le\sigma_{\eps}^{-1}$,
$M_{q}:=[M_{\eps q}]\in\hyperN{\sigma}$ because of \eqref{eq:Mbar-M_q-rel},
and
\begin{equation}
a_{n}:=\begin{cases}
[a_{\nint{(n)}_{\eps},\eps}] & \text{if }n\ge M_{1}\text{ in }\hyperN{\sigma}\\
1 & \text{otherwise}
\end{cases}\quad\forall n\in\hyperN{\sigma}.\label{eq:extending-a}
\end{equation}
We have to prove that this well-defines a hypersequence $(a_{n})_{n}:\hyperN{\sigma}\ra\rcrho$.
First of all, the sequence is well-defined with respect to the equality
in $\hyperN{\sigma}$ because of Lem.~\ref{lem:nearestInt}. Moreover,
setting $q=1$ in \eqref{eq:epsLim}, we get $\rho_{\eps}-l_{\eps}<a_{n,\eps}<\rho_{\eps}+l_{\eps}$
for all $\eps$ and for all $n\ge M_{\eps1}$. If $n\ge M_{1}$ in
$\hyperN{\sigma}$, then $\nint{(n)}_{\eps}\ge M_{\eps1}$ for $\eps$
small, and hence $\rho_{\eps}-l_{\eps}<a_{\nint{(n)}_{\eps},\eps}<\rho_{\eps}+l_{\eps}$.
This shows that $a_{n}\in\rcrho$ because we assumed that $l=[l_{\eps}]\in\rcrho$.
Finally, \eqref{eq:epsLim} and \eqref{eq:M_eps_q-ineq} yield that
if $n\ge M_{q}$ then $n\ge M_{1}$ and hence $|a_{n}-l|<\diff{\rho}^{q}$.
\end{proof}
\noindent From the proof it also follows, more generally, that if
$(M_{\eps q})_{\eps,q}$ satisfies \eqref{eq:epsLim} and if
\[
\exists(q_{\eps})\to+\infty:\ \left(M_{\eps,q_{\eps}}\right)\in\R_{\rho},
\]
then we can repeat the proof with $q_{\eps}$ instead of $\lceil\frac{1}{\eps}\rceil$
and setting $\sigma:=\rho$.

\subsection{Operations with hyperlimits and inequalities}

Thanks to Def.~\ref{def:setOfRadii} of sharp topology and our notation
for $x<y$ (and of the consequent Lem.~\ref{lem:mayer}), some results
about hyperlimits can be proved by trivially generalizing classical
proofs. For example, if $(x_{n})_{n\in\hyperN{\sigma}}$ and $(y_{n})_{n\in\hyperN{\sigma}}$
are two convergent hypersequences then their sum $(x_{n}+y_{n})_{n\in\hyperN{\sigma}}$,
product $(x_{n}\cdot y_{n})_{n\in\hyperN{\sigma}}$ and quotient $\left(\frac{x_{n}}{y_{n}}\right)_{n\in\hyperN{\sigma}}$
(the last one being defined only when $y_{n}$ is invertible for all
$n$$\in\hyperN{\sigma}$) are convergent hypersequences and the corresponding
hyperlimits are sum, product and quotient of the corresponding hyperlimits.

The following results generalize the classical relations between limits
and inequalities.
\begin{thm}
\label{thm:strLessSqueeze}Let $x$, $y$, $z:\hyperN{\sigma}\ra\rcrho$
be hypersequences, then we have:
\begin{enumerate}
\item \label{enu:ineqHyperlim}If \textup{$\hyperlim{\rho}{\sigma}x_{n}<\hyperlim{\rho}{\sigma}y_{n}$,
then $\exists M$$\in\hyperN{\sigma}$ such that $x_{n}<y_{n}$ for
all $n\geq M$, $n\in\hyperN{\sigma}$.}
\item \label{enu:squeeze}If $x_{n}\leq y_{n}\leq z_{n}$ for all $n\in\hyperN{\sigma}$
and \textup{$\hyperlim{\rho}{\sigma}x_{n}=\hyperlim{\rho}{\sigma}z_{n}=:l$,
then $\exists\,\hyperlim{\rho}{\sigma}y_{n}=l,$}
\end{enumerate}
\end{thm}

\begin{proof}
\ref{enu:ineqHyperlim} follows from Lem.~\ref{lem:mayer} and the
Def.~\ref{def:hyperlimit} of hyperlimit. For \ref{enu:squeeze},
the proof is analogous to the classical one. In fact, since $\hyperlim{\rho}{\sigma}x_{n}=\hyperlim{\rho}{\sigma}z_{n}=:l$
given $q\in\N$, there exist $M'$, $M''\in\hyperN{\sigma}$ such
that $l-\diff{\text{\ensuremath{\rho}}}^{q}<x_{n}$ and $z_{n}<l+\diff{\rho}^{q}$
for all $n>M'$, $n>M'',n\in\hyperN{\sigma}$, then for $n>M:=M'\vee M''$,
we have $l-\diff{\text{\ensuremath{\rho}}}^{q}<x_{n}\leq y_{n}\leq z_{n}<l+\diff{\text{\ensuremath{\rho}}}^{q}$.
\end{proof}
\begin{thm}
\label{thm:closed}Assume that $C$ is a sharply closed subset of
$\rcrho$, that $\exists\ \hyperlim{\rho}{\sigma}x_{n}=:l$ and that
$x_{n}$ eventually lies in $C$, i.e.~$\exists N\in\hyperN{\sigma}\,\forall n\in\hyperN{\sigma}_{\ge N}:\ x_{n}\in C$.
Then also $l\in C$. In particular, if $(y_{n})_{n}$ is another hypersequence
such that $\exists\,\hyperlim{\rho}{\sigma}y_{n}=:k$, then $\exists N\in\hyperN{\sigma}\,\forall n\in\hyperN{\sigma}_{\ge N}:\ x_{n}\ge y_{n}$
implies $l\geq k$.
\end{thm}

\begin{proof}
A reformulation of the usual proof applies. In fact, let us suppose
that $l\in\rcrho\setminus C$. Since $\rcrho\setminus C$ is sharply
open, there is an $\eta>0,$ for which $B_{\eta}(l)\text{\ensuremath{\subseteq} \ensuremath{\rcrho}\ensuremath{\setminus C}}$.
Let $\bar{n}\in\hyperN{\sigma}_{\ge N}$ be such that $|x_{n}-l|<\eta$
when $n>\bar{n}$. Then we have $x_{n}\in C$ and $x_{n}\in B_{\eta}(l)\subseteq\rcrho\setminus C$,
a contradiction.
\end{proof}
The following result applies to all generalized smooth functions (and
hence to all Colombeau generalized functions, see e.g.~\cite{GKV,TI};
see also \cite{AFJ} for a more general class of functions) because
of their continuity in the sharp topology.
\begin{thm}
\label{thm:cont}Suppose that $f:U\ra\rcrho$. Then $f$ is sharply
continuous function at $x=c$ if and only if it is hyper-sequentially
continuous, i.e.~for any hypersequence $\left(x_{n}\right)_{n}$
in $U$ converging to $c$, the hypersequence $\left(f\left(x_{n}\right)\right)_{n}$
converges to $f\left(c\right)$, i.e.~$f\left(\hyperlim{\rho}{\sigma}x_{n}\right)=\hyperlim{\rho}{\sigma}f(x_{n})$.
\end{thm}

\begin{proof}
We only prove that the hyper-sequential continuity is a sufficient
condition, because the other implication is a trivial generalization
of the classical one. By contradiction, assume that for some $Q\in\N$
\begin{equation}
\forall n\in\N\,\exists x_{n}\in U:\ |x_{n}-c|<\diff{\rho}^{n},\ |f(x_{n})-f(c)|>_{\text{s}}\diff{\rho}^{Q}.\label{eq:seqContContr}
\end{equation}
For $n\in\N$ set $\omega_{n}:=n$ and for $n\in\hypNr\setminus\N$
set $\omega_{n}:=\min\left\{ N\in\N\mid n\le\diff{\rho}^{-N}\right\} $
and $x_{n}:=x_{\omega_{n}}$. Then for all $n\in\hypNr$, from \eqref{eq:seqContContr}
we get $|x_{n}-c|<\diff{\rho}^{\omega_{n}}\to0$ because $\omega_{n}\to+\infty$
as $n\to+\infty$ in $n\in\hypNr$. Therefore, $(x_{n})_{n}$ is an
hypersequence of $U$ that converges to $c$, which yields $f(x_{n})\to f(c)$,
in contradiction with \eqref{eq:seqContContr}.
\end{proof}
\begin{example}
Let $\sigma\le\rho^{R}$ for some $R\in\R_{>0}$. The following inequalities
hold for all generalized numbers because they also hold for all real
numbers:
\begin{align}
\ln(x) & \le x\nonumber \\
e\left(\frac{n}{e}\right)^{n} & \le n!\le en\left(\frac{n}{e}\right)^{n}.\label{eq:sterl}
\end{align}
From the first one it follows $0\le\frac{\ln(n)}{n}=\frac{2\ln\sqrt{n}}{n}\le\frac{2\sqrt{n}}{n}$,
so that $\hyperlim{\rho}{\sigma}\frac{\ln(n)}{n}:=0$ from Thm\@.~\ref{thm:strLessSqueeze}
and $\hyperlim{\rho}{\sigma}n^{1/n}=1$ from Thm\@.~\ref{thm:cont}
and hence $\hyperlim{\rho}{\sigma}(n!)^{1/n}=+\infty$ by \eqref{eq:sterl}.
Similarly, we have $\underset{}{\hyperlim{\rho}{\sigma}}\left(1+\frac{1}{n}\right)^{n}=e$
because $n\log\left(1+\frac{1}{n}\right)=1-\frac{1}{2n}+O\left(\frac{1}{n^{2}}\right)\to1$
and because of Thm\@.~\ref{thm:cont}.
\end{example}

A little more involved proof concerns L'Hôpital rule for generalized
smooth functions. For the sake of completeness, here we only recall
the equivalent definition:
\begin{defn}
Let $X\subseteq\rcrho^{n}$ and $Y\subseteq\rcrho^{d}$. We say that
$f:X\longrightarrow Y$ is a \emph{generalized smooth function}\textcolor{red}{{}
}(GSF) if
\begin{enumerate}
\item $f:X\ra Y$ is a set-theoretical function.
\item There exists a net $(f_{\eps})\in\Coo(\R^{n},\R^{d})^{(0,1]}$ such
that for all $[x_{\eps}]\in X$:
\begin{enumerate}[label=(\alph*)]
\item $f(x)=[f_{\eps}(x_{\eps})]$
\item $\forall\alpha\in\N^{n}:\ (\partial^{\alpha}f_{\eps}(x_{\eps}))\text{ is }\rho-\text{moderate}$.
\end{enumerate}
\end{enumerate}
For generalized smooth functions lots of results hold: closure with
respect to composition, embedding of Schwartz's distributions, differential
calculus, one-di\-men\-sio\-nal integral calculus using primitives,
classical theorems (intermediate value, mean value, Taylor, extreme
value, inverse and implicit function), multidimensional integration,
Banach fixed point theorem, a Picard-Lindelöf theorem for both ODE
and PDE, several results of calculus of variations, etc.

In particular, we have the following (see also \cite{GK13b} for the
particular case of Colombeau generalized functions):
\end{defn}

\begin{thm}
\label{thm:FR-forGSF} Let $U\subseteq\rcrho$ be a sharply open set
and let $f:U\ra\rcrho$ be a GSF defined by the net of smooth functions
$f_{\eps}\in\cinfty(\R,\R)$. Then
\begin{enumerate}
\item \label{enu:existenceRatio}There exists an open neighbourhood $T$
of $U\times\{0\}$ and a GSF $R_{f}:T\to\rcrho$, called the \emph{generalized
incremental ratio} of $f$, such that 
\begin{equation}
f(x+h)=f(x)+h\cdot R_{f}(x,h)\qquad\forall(x,h)\in T.\label{eq:FR}
\end{equation}
Moreover $R_{f}(x,0)=\left[f'_{\eps}(x_{\eps})\right]=f'(x)$ is another
GSF and we can hence recursively define $f^{(k)}(x)$.
\item \label{enu:uniquenessRatio}Any two generalized incremental ratios
of $f$ coincide on the intersection of their domains.
\item More generally, for all $k\in\N_{>0}$ there exists an open neighbourhood
$T$ of $U\times\{0\}$ and a GSF $R_{f}^{k}:T\to\rcrho$, called
\emph{$k$-th order Taylor ratio }of $f$, such that
\begin{equation}
f(x+h)=\sum_{j=0}^{k-1}\frac{f^{(j)}(x)}{j!}h^{j}+R_{f}^{k}(x,h)\cdot h^{k}\qquad\forall(x,h)\in T.\label{eq:Peano}
\end{equation}
Any two ratios of $f$ of the same order coincide on the intersection
of their domains.
\end{enumerate}
\end{thm}

\noindent We can now prove the following generalization of one of
L'Hôpital rule:
\begin{thm}
\label{thm:LHopital}Let $U\subseteq\rti$ be a sharply open set $(x_{n})_{n}$,
$(y_{n})_{n}:\hypNs\ra U$ be hypersequences converging to $l\in U$
and $m\in U$ respectively and such that
\[
\hyperlim{\rho}{\sigma}\frac{x_{n}-l}{y_{n}-m}=:C\in\rcrho.
\]
Let $k\in\N_{>0}$ and $f$, $g:U\ra\rcrho$ be GSF such that for
all $n\in\hypNs$ and all $j=0,\ldots,k-1$ 
\begin{align}
 & g^{(j)}(y_{n})\in\rcrho^{*}\nonumber \\
 & f^{(j)}(l)=g^{(j)}(m)=0\label{eq:zero}\\
 & g^{(k)}(m)\in\rcrho^{*}\nonumber 
\end{align}
Then for all $j=0,\ldots,k-1$
\[
\exists\,\hyperlim{\rho}{\sigma}\frac{f^{(j)}(x_{n})}{g^{(j)}(y_{n})}=C^{k}\cdot\hyperlim{\rho}{\sigma}\frac{f^{(k)}(x_{n})}{g^{(k)}(y_{n})}.
\]
\end{thm}

\begin{proof}
Using \eqref{eq:Peano} and \eqref{eq:zero}, we can write
\begin{multline*}
\frac{f(x_{n})}{g(y_{n})}=\frac{\sum_{j=0}^{k-1}\frac{f^{(j)}(l)}{j!}(x_{n}-l)^{j}+(x_{n}-l)^{k}R_{f}^{k}(l,x_{n}-l)}{\sum_{j=0}^{k-1}\frac{g^{(j)}(m)}{j!}(y_{n}-m)^{j}+(y_{n}-m)^{k}R_{g}^{k}(m,y_{n}-m)}=\\
=\left(\frac{x_{n}-l}{y_{n}-m}\right)^{k}\cdot\frac{R_{f}^{k}(l,x_{n}-l)}{R_{g}^{k}(m,y_{n}-m)}.
\end{multline*}
Since $R_{f}^{k}$ and $R_{g}^{k}$ are GSF, they are sharply continuous.
Therefore, the right hand side of the previous equality tends to $C^{k}\cdot\frac{R_{f}^{k}(l,0)}{R_{g}^{k}(m,0)}=C^{k}\cdot\frac{f^{(k)}(l)}{g^{(k)}(m)}$.
At the same limit converges the quotient $C^{k}\frac{f^{(k)}(x_{n})}{g^{(k)}(y_{n})}$
because $f^{(k)}$ and $g^{(k)}$ are also GSF and hence they are
sharply continuous. The claim for $j=1,\ldots,k-1$ follows by applying
the conclusion for $j=0$ with $f^{(j)}$ and $g^{(j)}$ instead of
$f$ and $g$.
\end{proof}
\noindent Note that for $x_{n}=y_{n}$, $l=m$, we have $C=1$ and
we get the usual L'Hôpital rule (formulated using hypersequences).
Note that a similar theorem can also be proved without hypersequences
and using the same Taylor expansion argument as in the previous proof.

\subsection{Cauchy criterion and monotonic hypersequences.}

In this section, we deal with classical criteria implying the existence
of a hyperlimit.
\begin{defn}
We say that $(x_{n})_{n\in\hyperN{\sigma}}$ is a \emph{Cauchy hypersequence}
if
\end{defn}

\[
\text{\ensuremath{\forall}}q\in\mathbb{N}\,\exists M\in\hyperN{\sigma}\,\text{\ensuremath{\forall}}n,m\in\hyperN{\sigma}_{\ge M}:\ |x_{n}-x_{m}|<\diff{\rho^{q}}.
\]

\begin{thm}
\label{thm:Cauchy}A hypersequence converges if and only if it is
a Cauchy hypersequence
\end{thm}

\begin{proof}
To prove that the Cauchy criterion is a necessary condition it suffices
to consider the inequalities:
\[
|x_{n}-x_{m}|\le|x_{n}-l|+|x_{m}-l|\le\diff{\rho^{q+1}}+\diff{\rho^{q+1}}<\diff{\rho}^{q}
\]
Vice versa, assume that
\begin{equation}
\text{\ensuremath{\forall}}q\in\mathbb{N}\,\exists M_{q}\in\hyperN{\sigma}\,\text{\ensuremath{\forall}}n,m\in\hyperN{\sigma}_{\ge M_{q}}:\ |x_{n}-x_{m}|<\diff{\rho^{q}}.\label{eq:Cauchy}
\end{equation}
The idea is to use Cauchy completeness of $\rcrho$. In fact, set
$h_{1}:=M_{1}$ and $h_{q+1}:=M_{q+1}\vee h_{q}$. We claim that $(x_{h_{q}})_{q\in\mathbb{N}}$
is a standard Cauchy sequence converging to the same limit of $(x_{n})_{n\in\hyperN{\sigma}}$
. From \eqref{eq:Cauchy} it follows that $(x_{h_{q}})_{q\in\mathbb{N}}$
is a standard Cauchy sequence (in the sharp topology). Therefore,
there exists $\bar{x}\in\rcrho$ such that $\lim_{q\to+\infty}x_{h_{q}}=\bar{x}$.
Now, fix $q\in\mathbb{N}$ and pick any $m\ge q+1$ such that 
\begin{equation}
|x_{h_{m}}-\bar{x}|<\diff{\rho^{q+1}}.\label{eq:CauchyConv}
\end{equation}
Then for all $N\ge M_{q+1}$ we have: 
\[
|x_{N}-\bar{x}|\le|x_{N}-x_{h_{m}}|+|x_{h_{m}}-\bar{x}|<2\diff{\rho^{q+1}}<\diff{\rho^{q}}
\]
because $h_{m}\ge h_{q+1}\ge M_{q+1}$ so that we can apply \eqref{eq:Cauchy}
and \eqref{eq:CauchyConv}.
\end{proof}
\begin{thm}
\label{thm:Cauchy_m_ge_n}A hypersequence converges if and only if
\[
\text{\ensuremath{\forall}}q\in\mathbb{N}\,\exists M\in\hyperN{\sigma}\,\text{\ensuremath{\forall}}n,m\in\hyperN{\sigma}_{\ge M}:m\ge n\ \Rightarrow\ |x_{n}-x_{m}|<\diff{\rho^{q}}.
\]
\end{thm}

\begin{proof}
It suffices to apply the inequality $\left|x_{n}-x_{m}\right|\le\left|x_{n}-x_{n\vee m}\right|+\left|x_{n\vee m}-x_{m}\right|$.
\end{proof}
The second classical criterion for the existence of a hyperlimit is
related to the notion of monotonic hypersequence. The existence of
several chains in $\hypNs$ does not allow to arrive at any $M\in\hypNs$
starting from any other lower $N\in\hypNs$ and using the successor
operation only a finite number of times. For this reason, the following
is the most natural notion of monotonic hypersequence:
\begin{defn}
\label{def:monotonicity}We say that $(x_{n})_{N\in\hyperN{\sigma}}$
is a \emph{non-decreasing} (or \emph{increasing}) hypersequence if
\[
\text{\ensuremath{\forall}}n,m\in\hypNs:\ n\ge m\ \Rightarrow\ x_{n}\ge x_{m}.
\]
Similarly, we can define the notion of non-increasing (decreasing)
hypersequence.
\end{defn}

\begin{thm}
\label{thm:monotonicHyperseq}Let $(x_{n})_{n}:\hypNs\ra\rcrho$ be
a non-decreasing hypersequence. Then
\[
\exists\,\hyperlim{\rho}{\sigma}x_{n}\ \iff\ \exists\,\sup\left\{ x_{n}\mid n\in\hypNs\right\} ,
\]
and in that case they are equal. In particular, if $\left\{ x_{n}\mid n\in\hypNs\right\} $
is complete from above for all the upper bounds, then
\[
\exists\,\hyperlim{\rho}{\sigma}x_{n}\ \iff\ \exists U\in\rcrho\,\forall n\in\hypNs:\ x_{n}\le U.
\]
\end{thm}

\begin{proof}
Assume that $(x_{n})_{n\in\hyperN{\sigma}}$ converges to $l$ and
set $S:=\{x_{n}\mid n\in\hypNs\}$, we will show that $l=\sup(S)$.
Now, using Def.~\ref{def:hyperlimit}, we have that $\forall n\in\hypNs_{\ge N}:\ x_{n}<l+\diff{\rho}^{q}$
for some $N\in\hypNs$. But from Def.~\ref{def:monotonicity} $\forall n\in\hypNs:\ x_{n}\le x_{n\vee N}<l+\diff{\rho}^{q}$.
Therefore $x_{n}\le l+\diff{\rho}^{q}$ for all $n\in\hypNs$, and
the conclusion $x_{n}\le l$ follows since $q\in\N$ is arbitrary.
Finally, from Def.~\ref{def:hyperlimit} of hyperlimit, for all $q\in\N$
we have the existence of $L\in\hyperN{\sigma}$ such that $l-\diff{\rho^{q}}<x_{L}\in S$
which completes the necessity part of the proof. Now, assume that
$\exists\,\sup(S)=:l$. We have to prove that $\hyperlim{\rho}{\sigma}x_{n}=l$.
In fact, using Rem.~\ref{rem:sharpSupremum}, we get
\[
\forall q\in\mathbb{N}\,\exists x_{N}\in S:\ l-\diff{\rho^{q}}<x_{N},
\]
and $x_{N}\le x_{n}\le l<l+\diff{\rho^{q}}$ for all $n\in\hypNs_{\ge N}$
by Def.~\ref{def:monotonicity} of monotonicity. That is, $\left|l-x_{n}\right|=x_{n}-l<\diff{\rho}^{q}$.
\end{proof}
\begin{example}
\label{exa:rootsInf}The hypersequence $x_{n}:=\diff{\rho}^{1/n}$
is non-decreasing. Assume that $(x_{n})_{n}$ converges to $l$ and
that $\sigma\le\rho^{R}$ for some $R\in\R_{>0}$. Since $x_{n}\ge\diff{\rho}$,
by Thm.~\ref{thm:closed}, we get $l\ge\diff{\rho}$. Therefore,
applying the logarithm and the exponential functions, from Thm.~\ref{thm:cont}
we obtain that $l=1$ because from $\sigma\le\rho^{R}$ it follows
that $\hyperlim{\rho}{\sigma}\frac{\log(\diff{\rho})}{n}=0$. But
this is impossible since $1\approx1-\diff{\rho}\nleq\diff{\rho}^{1/n}$.
Thereby, $\nexists\,\sup\left\{ \diff{\rho}^{1/n}\mid n\in\hypNs_{>0}\right\} $
and this set is also not complete from above.
\end{example}

\section{Limit superior and inferior}

We have two possibilities to define the notions of limit superior
and inferior in a non-Archimedean setting such as $\rcrho$: the first
one is to assume that both $\alpha_{m}:=\sup\{x_{n}\mid n\in\hypNs_{\ge m}\}$
and $\inf\{\alpha_{m}\mid m\in\hypNs\}$ exist (the former for all
$m\in\hypNs$); the second possibility is to use inequalities to avoid
the use of supremum and infimum. In fact, in the real case we have
$\iota\le\sup_{n\ge m}x_{n}\le\iota+\eps$ if and only if
\begin{align*}
 & \forall n\ge m:\ x_{n}\le\iota+\eps\\
 & \forall\eps\,\exists\bar{n}\ge m:\ \iota-\eps\le x_{\bar{n}}.
\end{align*}

\begin{defn}
\label{def:limsup-inf}Let $(x_{n})_{n}:\hypNs\ra\rcrho$ be an hypersequence,
then we say that $\iota\in\rcrho$ is the \emph{limit superior} of
$(x_{n})_{n}$ if
\begin{enumerate}
\item \label{enu:1limsup}$\forall q\in\N\,\exists N\in\hypNs\,\forall n\ge N:\ x_{n}\le\iota+\diff{\rho}^{q}$;
\item \label{enu:2limsup}$\forall q\in\N\,\forall N\in\hypNs\,\exists\bar{n}\ge N:\ \iota-\diff{\rho}^{q}\le x_{\bar{n}}$.
\end{enumerate}
\noindent Similarly, we say that $\sigma\in\rcrho$ is the \emph{limit
inferior} of $(x_{n})_{n}$ if
\begin{enumerate}[resume]
\item \label{enu:1liminf}$\forall q\in\N\,\exists N\in\hypNs\,\forall n\ge N:\ x_{n}\ge\sigma-\diff{\rho}^{q}$;
\item \label{enu:2liminf}$\forall q\in\N\,\forall N\in\hypNs\,\exists\bar{n}\ge N:\ \sigma+\diff{\rho}^{q}\ge x_{\bar{n}}$.
\end{enumerate}
\end{defn}

We have the following results (clearly, dual results hold for the
limit inferior):
\begin{thm}
\label{thm:limsup-inf}Let $(x_{n})_{n}$, $(y_{n})_{n}:\hypNs\ra\rcrho$
be hypersequences, then
\begin{enumerate}
\item \label{enu:limsupUnique}There exists at most one limit superior and
at most one limit inferior. They are denoted with $\hyplimsup{\rho}{\sigma}x_{n}$
and $\hypliminf{\rho}{\sigma}x_{n}$.
\item \label{enu:limsupClassicalDef}If $\exists\sup\left\{ x_{n}\mid n\in\hypNs_{\ge m}\right\} =:\alpha_{m}$
for all $m\in\hypNs$, then $\exists\,\hyplimsup{\rho}{\sigma}x_{n}$
if and only if $\exists\inf\left\{ \alpha_{m}\mid m\in\hypNs\right\} $,
and in that case
\[
\hyplimsup{\rho}{\sigma}x_{n}=\hyperlimarg{\rho}{\sigma}{m}\alpha_{m}=\inf\left\{ \alpha_{m}\mid m\in\hypNs\right\} .
\]
\item \label{enu:limsupLimInf}$\hyplimsup{\rho}{\sigma}(-x_{n})=-\hypliminf{\rho}{\sigma}x_{n}$
in the sense that if one of them exists, then also the other one exists
and in that case they are equal.
\item \label{enu:limsupLim}$\exists\hyperlim{\rho}{\sigma}x_{n}$ if and
only if $\exists\hyplimsup{\rho}{\sigma}x_{n}=\hypliminf{\rho}{\sigma}x_{n}$.
\item \label{enu:limsupSuperAdd}If $\exists\hyplimsup{\rho}{\sigma}x_{n}$,
$\hyplimsup{\rho}{\sigma}y_{n}$, $\hyplimsup{\rho}{\sigma}(x_{n}+y_{n})$,
then
\[
\hyplimsup{\rho}{\sigma}(x_{n}+y_{n})\le\hyplimsup{\rho}{\sigma}x_{n}+\hyplimsup{\rho}{\sigma}y_{n}.
\]
In particular, if $\forall N\in\hypNs\,\forall\bar{n},\hat{n}\ge N\,\exists n\ge N:\ x_{\bar{n}}+y_{\hat{n}}\le x_{n}+y_{n}$,
then the existence of the single limit superiors implies the existence
of the limit superior of the sum.
\item \label{enu:limsupSuperProd}If $x_{n}$, $y_{n}\ge0$ for all $n\in\hypNs$
and if $\exists\hyplimsup{\rho}{\sigma}x_{n}$, $\hyplimsup{\rho}{\sigma}y_{n}$,
$\hyplimsup{\rho}{\sigma}(x_{n}\cdot y_{n})$, then
\[
\hyplimsup{\rho}{\sigma}(x_{n}\cdot y_{n})\le\hyplimsup{\rho}{\sigma}x_{n}\cdot\hyplimsup{\rho}{\sigma}y_{n}.
\]
In particular, if $\forall N\in\hypNs\,\forall\bar{n},\hat{n}\ge N\,\exists n\ge N:\ x_{\bar{n}}\cdot y_{\hat{n}}\le x_{n}\cdot y_{n}$,
then the existence of the single limit superiors implies the existence
of the limit superior of the product.
\item \label{enu:limsupSubseqNec}If $\exists\,\hyplimsup{\rho}{\sigma}x_{n}=:\iota$,
then there exists a sequence $(\bar{n}_{q})_{q\in\N}$ of $\hypNs$
such that
\begin{enumerate}[label=(\alph*)]
\item \label{enu:limsupSub_a}$\bar{n}_{q+1}>\bar{n}_{q}$ for all $q\in\N$;
\item \label{enu:limsupSub_b}$\lim_{q\to+\infty}\bar{n}_{q}=+\infty$ in
$\RC{\sigma}$;
\item \label{enu:limsupSub_c}$\exists\lim_{q\to+\infty}x_{\bar{n}_{q}}=\iota$.
\end{enumerate}
\item \label{enu:limsupSubseqSuff}Assume to have a sequence $(\bar{n}_{q})_{q\in\N}$
satisfying the previous conditions \ref{enu:limsupSub_a}, \ref{enu:limsupSub_b},
\ref{enu:limsupSub_c} and
\begin{equation}
\forall n\in\hypNs\,\exists p\in\N:\ \bar{n}_{p}\ge n,\ x_{n}\le x_{\bar{n}_{p}}.\label{eq:seqSuff}
\end{equation}
Then $\exists\,\hyplimsup{\rho}{\sigma}x_{n}=:\iota$.
\end{enumerate}
\end{thm}

\begin{proof}
\ref{enu:limsupUnique}: Let $\iota_{1}$, $\iota_{2}$ be both limit
superior of $(x_{n})_{n}$. Based on Lem.~\ref{lem:trich1st}.\ref{enu:trichotomy},
without loss of generality we can assume that $\iota_{1}\sbpt{<}\iota_{2}$.
According to Lem.~\ref{lem:mayer}, there exists $m\in\N$ such that
$\iota_{1}+\diff{\rho}^{m}\sbpt{<}\iota_{2}$. Take $q_{1}$, $q_{2}$
large enough so that $\diff{\rho}^{q_{1}}+\diff{\rho}^{q_{2}}<\diff{\rho}^{m}$.
Using the last two inequalities, we obtain
\begin{equation}
\iota_{1}+\diff{\rho}^{q_{1}}\sbpt{<}\iota_{2}-\diff{\rho}^{q_{2}}.\label{eq:q1q2}
\end{equation}
Using Def.~\ref{def:limsup-inf}.\ref{enu:1limsup}, we can find
$N_{1}\in\hypNs$ such that
\begin{equation}
\forall n\in\hypNs_{\ge N_{1}}:\ x_{n}\le\iota_{1}+\diff{\rho}^{q_{1}}.\label{eq:i_iota1}
\end{equation}
Using Def.~\ref{def:limsup-inf}.\ref{enu:2limsup} with $q=q_{1}$
and $N=N_{1}$, we get
\begin{equation}
\exists\bar{n}\in\hypNs_{\ge N_{1}}:\ \iota_{2}-\diff{\rho}^{q_{2}}\le x_{\bar{n}}.\label{eq:ii_iota2}
\end{equation}
We now use \eqref{eq:q1q2}, \eqref{eq:ii_iota2} and \eqref{eq:i_iota1}
for $n=\bar{n}$ and we obtain $x_{\bar{n}}\le\iota_{1}+\diff{\rho}^{q_{1}}\sbpt{<}\iota_{2}-\diff{\rho}^{q_{2}}\le x_{\bar{n}}$,
which is a contradiction.

\ref{enu:limsupClassicalDef}: Lem.~\ref{lem:propSupInf}.\ref{enu:incr}
implies that $(\alpha_{m})_{m}$ is non-increasing. Therefore, we
have $\hyperlim{\rho}{\sigma}\alpha_{m}=\inf\left\{ \alpha_{m}\mid m\in\hypNs\right\} $
if these terms exist from Thm.~\ref{thm:monotonicHyperseq}. But
Cor.~\ref{cor:supLUB} and Def.~\ref{def:limsup-inf}.\ref{enu:1limsup}
imply $\alpha_{m}\le\iota+\diff{\rho}^{q}$. Finally, Def.~\ref{def:limsup-inf}.\ref{enu:2limsup}
yields $\iota-\diff{\rho}^{q}\le x_{\bar{n}}\le\alpha_{m}$, which
proves that $\exists\,\hyperlim{\rho}{\sigma}\alpha_{m}=\hyplimsup{\rho}{\sigma}x_{m}=\iota$.

\ref{enu:limsupLimInf}: Directly from Def.~\ref{def:limsup-inf}.

\ref{enu:limsupLim}: Assume that hyperlimit superior and inferior
exist and are equal to $l$. From Def.~\ref{def:limsup-inf}.\ref{enu:1limsup}
and Def.~\ref{def:limsup-inf}.\ref{enu:1liminf} we get $l-\diff{\rho}^{q}\le x_{n}\le l+\diff{\rho}^{q}$
for all $n\ge N$. Vice versa, assume that the hyperlimit exists and
equals $l$, so that $l-\diff{\rho}^{q}\le x_{n}\le l+\diff{\rho}^{q}$
for all $n\ge N$. Then both Def.~\ref{def:limsup-inf}.\ref{enu:1limsup}
and Def.~\ref{def:limsup-inf}.\ref{enu:1liminf} trivially hold.
Finally, Def.~\ref{def:limsup-inf}.\ref{enu:2limsup} and Def.~\ref{def:limsup-inf}.\ref{enu:2liminf}
hold taking e.g.~$\bar{n}=N$.

\ref{enu:limsupSuperAdd}: Setting
\begin{align*}
\iota & :=\hyplimsup{\rho}{\sigma}x_{n}\\
j & :=\hyplimsup{\rho}{\sigma}y_{n}\\
l & :=\hyplimsup{\rho}{\sigma}\left(x_{n}+y_{n}\right),
\end{align*}
from Def.~\ref{def:limsup-inf} we get $l-\diff{\rho}^{q}\le x_{\bar{n}}+y_{\bar{n}}\le\iota+j+2\diff{\rho}^{q}$,
which implies $l\le\iota+j$ for $q\to+\infty$. Adding Def.~\ref{def:limsup-inf}.\ref{enu:2limsup}
we obtain $\iota+j-2\diff{\rho}^{q}\le x_{\bar{n}}+y_{\hat{n}}$ for
some $\bar{n}$, $\hat{n}\ge N\in\hypNs$. Therefore, if $x_{\bar{n}}+y_{\hat{n}}\le x_{n}+y_{n}$
for some $n\ge N$, this yields the second claim. Similarly, one can
prove \ref{enu:limsupSuperProd}.

\ref{enu:limsupSubseqNec}: From Def.~\ref{def:limsup-inf}.\ref{enu:1limsup},
choose an $N_{q}=N$ for each $q\in\N$, i.e.
\begin{equation}
\forall q\in\N\,\exists N_{q}\in\hypNs\,\forall n\ge N_{q}:\ x_{n}\le\iota+\diff{\rho}^{q}.\label{eq:seq}
\end{equation}
Applying Def.~\ref{def:limsup-inf}.\ref{enu:2limsup} with $q>0$
and $N=N_{q}\vee\left(\bar{n}_{q-1}+1\right)\vee\left[\text{int}(\sigma_{\eps}^{-q})\right]\in\hypNs$,
we get the existence of $\bar{n}_{q}\ge N_{q}$ such that both \ref{enu:limsupSub_a}
and \ref{enu:limsupSub_b} hold and $\iota-\diff{\rho}^{q}\le x_{\bar{n}_{q}}$.
Thereby, from \eqref{eq:seq} we also get \ref{enu:limsupSub_c}.

\ref{enu:limsupSubseqSuff}: Write \ref{enu:limsupSub_c} as
\begin{equation}
\forall q\in\N\,\exists Q_{q}\in\N\,\forall p\in\N_{\ge Q_{q}}:\ \iota-\diff{\rho}^{p}\le x_{\bar{n}_{p}}\le\iota+\diff{\rho}^{p}.\label{eq:sublim}
\end{equation}
Set $N:=\bar{n}_{Q_{q}}\in\hypNs$. For $n\ge N$, from \eqref{eq:seqSuff}
we get the existence of $p\in\N$ such that $\bar{n}_{p}\ge n$ and
$x_{n}\le x_{\bar{n}_{p}}$. Thereby, $\bar{n}_{p}\ge\bar{n}_{Q_{q}}$
and hence $p\ge Q_{q}$ because of \ref{enu:limsupSub_a} and thus
$x_{n}\le x_{\bar{n}_{p}}\le\iota+\diff{\rho}^{q}$. Finally, condition
\ref{enu:2limsup} of Def.~\ref{def:limsup-inf} follows from \eqref{eq:sublim}
and \ref{enu:limsupSub_b}.
\end{proof}
It remains an open problem to show an example that proves as necessary
the assumption of Thm\@.~\ref{thm:limsup-inf}.\ref{enu:limsupClassicalDef},
i.e.~that that the previous definition of limit superior and inferior
is strictly more general than the simple transposition of the classical
one.
\begin{example}
\label{exa:limisupinf}~
\begin{enumerate}
\item Directly from Def.~\ref{def:limsup-inf}, we have that
\[
\hyplimsup{\rho}{\sigma}(-1)^{n}=1,\ \hypliminf{\rho}{\sigma}(-1)^{n}=-1
\]
\item Let $\mu\in\rcrho$ be such that $\mu|_{L}=1$ and $\mu|_{L^{c}}=-1$,
where $L$, $L^{c}\subzero I$. Then $\mu^{n}\le1$ and $1-\diff{\rho}^{q}\le\mu^{\bar{n}}$
if $\nint{(\bar{n})}_{\eps}$ is even for all $\eps$ small. Therefore
$\hyplimsup{\rho}{\sigma}\mu^{n}=1$, $\sup_{n\ge m}\mu^{n}=1$, whereas
$\nexists\,\hyperlim{\rho}{\sigma}\mu^{n}$.
\item From \ref{enu:limsupSubseqNec} and \ref{enu:limsupSubseqSuff} of
Thm.~\ref{thm:limsup-inf} it follows that for an increasing hypersequence
$(x_{n})_{n}$, $\exists\hyplimsup{\rho}{\sigma}x_{n}$ if and only
if $\exists\hyperlim{\rho}{\sigma}x_{n}$. Therefore, example \ref{exa:rootsInf}
implies that $\nexists\,\hyplimsup{\rho}{\sigma}\diff{\rho}^{1/n}$.
\end{enumerate}
\end{example}

\section{Conclusions}

In this work we showed how to deal with several deficiencies of the
ring of Robinson-Colombeau generalized numbers $\rcrho$: trichotomy
law for the order relations $\le$ and $<$, existence of supremum
and infimum and limits of sequences with a topology generated by infinitesimal
radii. In each case, we obtain a faithful generalization of the classical
case of real numbers. We think that some of the ideas we presented
in this article can inspire similar works in other non-Archimedean
settings such as (constructive) nonstandard analysis, p-adic analysis,
the Levi-Civita field, surreal numbers, etc. Clearly, the notions
introduced here open the possibility to extend classical proofs in
dealing with series, analytic generalized functions, sigma-additivity
in integration of generalized functions, non-Archimedean functional
analysis, just to mention a few.
\begin{description}
\item [{Acknowledgments}] The authors would like to thank L.~Luperi Baglini
for the proof of Thm.~\ref{thm:Cauchy}, H.~Vernaeve for several
suggestions concerning Sec.~\ref{subsec:Archimedean-upper-bound},
D.E.~Kebiche for an improvement of Thm.~\ref{thm:limsup-inf}, and
the referee for several suggestions that have led to considerable
improvements of the paper.
\end{description}

\end{document}